\documentclass[a4paper,fleqn,longmktitle]{cas-sc}

\usepackage[numbers]{natbib}

\usepackage{graphicx}%
\usepackage{multirow}%
\usepackage{amsmath,amssymb,amsfonts}%
\usepackage{amsthm}%
\usepackage{mathrsfs}%
\usepackage[title]{appendix}%
\usepackage{xcolor}%
\usepackage{textcomp}%
\usepackage{manyfoot}%
\usepackage{cleveref}
\usepackage{booktabs}%
\usepackage{algorithm}%
\usepackage{algorithmicx}%
\usepackage{algpseudocode}%
\usepackage{listings}%
\usepackage{enumitem}
\usepackage{float}

\theoremstyle{plain}
\newtheorem{theorem}{Theorem}[section]
\newtheorem{proposition}[theorem]{Proposition}
\newtheorem{lemma}[theorem]{Lemma}

\theoremstyle{definition}

\newtheorem{remark}[theorem]{Remark}

\numberwithin{equation}{section}

\usepackage{tabularx}

\def\tsc#1{\csdef{#1}{\textsc{\lowercase{#1}}\xspace}}
\tsc{WGM}
\tsc{QE}

\ExplSyntaxOn
\cs_if_exist:NF \vbox_unpack_clear:N
  {
    \cs_new_protected:Npn \vbox_unpack_clear:N #1
      {
        \vbox_unpack:N #1
        \box_clear:N #1
      }
  }
\ExplSyntaxOff

 

\ExplSyntaxOn
\cs_gset:Npn \__first_footerline:
  {
    \group_begin:
    \small \sffamily
    \hfill \thepage \hfill 
    \group_end:
  }
\ExplSyntaxOff

\usepackage{fancyhdr}
\begin{document}

\shorttitle{Computational Oncology of Chemotaxis-Driven Tumour--Immune Spatial Patterning and Stability}

\title[mode = title]{Computational Oncology of Chemotaxis-Driven Tumour--Immune Spatial Patterning and Stability}

\author[1]{Zonghao Liu}
\fnmark[1]
\ead{liuzonghao42@gmail.com}

\author[2]{Jiguang Yu}
\fnmark[1]
\ead{jyu678@bu.edu}

\author[3,4]{Lei Su}
\fnmark[1]
\ead{sulei2023@ia.ac.cn}

\author[5]{Louis Shuo Wang}
\cormark[1]
\fnmark[1]
\ead{wang.s41@northeastern.edu}

\author[3]{Yang Du}
\ead{yang.du@ia.ac.cn}

\author[1,6,7]{Jingfeng Liu}
\cormark[1]
\ead{drjingfeng@126.com}

\affiliation[1]{organization={Department of Hepatopancreatobiliary Surgery, Clinical Oncology School of Fujian Medical University},
            city={Fuzhou},
            postcode={350014},
            state={Fujian},
            country={China}}

\affiliation[2]{organization={College of Engineering, Boston University},
            city={Boston},
            postcode={02215},
            state={MA},
            country={United States}}

\affiliation[3]{organization={CAS Key Laboratory of Molecular Imaging, Institute of Automation, Chinese Academy of Sciences},
            city={Beijing},
            postcode={100190},
            country={China}}

\affiliation[4]{organization={University of Chinese Academy of Sciences},
            city={Beijing},
            postcode={100190},
            country={China}}

\affiliation[5]{organization={Department of Mathematics, Northeastern University},
            city={Boston},
            postcode={02115},
            state={MA},
            country={United States}}

\affiliation[6]{organization={NHC Key Laboratory of Cancer Metabolism},
            city={Fuzhou},
            postcode={350014},
            state={Fujian},
            country={China}}

\affiliation[7]{organization={Fujian Key Laboratory of Advanced Technology for Cancer Screening and Early Diagnosis},
            city={Fuzhou},
            postcode={350014},
            state={Fujian},
            country={China}}

\cortext[1]{Corresponding author}

\fntext[1]{These authors contributed equally to this work as co-first authors.}

\begin{abstract}
We develop a reaction--diffusion--chemotaxis model for spatial tumour--immune--chemokine dynamics that couples logistic tumour growth, immune-mediated killing, chemokine-dependent immune recruitment, chemotactic migration, and signal production. For the nondimensional system, we establish local classical solvability, nonnegativity, a uniform tumour-density bound, and global mass estimates for the immune and chemokine components. The tumour-free equilibrium is stable precisely when the baseline immune-control index satisfies \(\sigma_0/\delta>1\), whereas positive homogeneous coexistence is characterized by a scalar nonlinear equation. Linearization in the Neumann Laplacian eigenbasis yields a mode-dependent cubic dispersion relation, showing that chemotaxis does not alter the tumour-invasion threshold but can destabilize homogeneous coexistence through a finite-wavelength oscillatory instability above a critical sensitivity \(\xi_c\). A conservative finite-volume discretization with upwind chemotactic fluxes and implicit backward differentiation formula time integration is used to test these predictions. Numerical experiments recover the analytical equilibria and growth rates, identify the dominant unstable mode, reproduce the transition to spatial heterogeneity, and quantify the effects of immune recruitment, decay, and diffusion on the stability boundary. Grid-refinement, mass-balance, residual, and nonnegativity diagnostics support the computational reliability of the results.
\end{abstract}


\maketitle

\section{Introduction}\label{sec:introduction}

\subsection{Motivation: computational tumour modelling}
\label{subsec:motivation-computational-tumour-modelling}

Cancer progression is governed by interacting processes across molecular,
cellular, and tissue scales \cite{beeghly2023measuring,desoyer2025computational}. At the molecular scale, cytokines, chemokines,
growth factors, and metabolic signals regulate proliferation, death, migration,
and immune activation \cite{tiwari2025molecular}. At the cellular scale, tumour cells, immune effector
cells, stromal cells, and endothelial cells interact through direct contact and
through diffusible mediators \cite{kaminska2015role,joshi2026dynamic,he2022extracellular}. 
At the tissue scale, these local interactions produce spatially heterogeneous patterns of tumour expansion, immune infiltration, immune exclusion, necrosis, invasion fronts, and treatment
response \cite{giuliani2025immune,biswas2022inference,liang2026separation}. Therefore, tumour progression is not only a temporal growth process;
it is also a spatially structured dynamical process.

Spatial mathematical models are particularly useful for describing these
mechanisms because they connect local cell--signal interactions with
macroscopic tumour behaviour. 
Continuum models based on reaction--diffusion
equations describe cell proliferation, death, signal production, degradation,
and passive spreading \cite{kondo2010reaction}. 
Chemotaxis terms describe directed cell migration along
chemical gradients, which is especially relevant when immune cells respond to
chemokine fields generated inside the tumour microenvironment \cite{mempel2024chemokines,liu2025bidirectional}. 
In this way, reaction--diffusion--chemotaxis systems provide a mechanistic framework for
studying how microscopic signalling and cell movement generate macroscopic
invasion, immune infiltration, immune exclusion, or coexistence patterns
\cite{matzavinos2004mathematical,tao2024global,wang2026elliptic,li2026global,ai2015reaction,ke2022analysis,kiselev2022chemotaxis}.

From a computational viewpoint, such models also provide a bridge between
mathematical analysis and numerical simulation. A biologically meaningful tumour
model should preserve nonnegativity of cell densities and chemical
concentrations, respect no-flux tissue boundaries when the model is posed on a
closed tissue region, and produce interpretable stability or instability
criteria. These requirements are not merely technical. Negative cell densities
are biologically meaningless, artificial boundary fluxes may distort tumour
growth or immune infiltration, and unstable numerical schemes may create
patterns that are artifacts of the discretization rather than consequences of
the model.

This work contributes to the development of mathematically analyzable and
numerically reliable partial differential equation (PDE) models for tumour-related
phenomena. The emphasis is not only on simulating spatial tumour--immune
dynamics, but also on identifying stability thresholds and designing numerical
diagnostics that preserve the biological constraints of nonnegativity and
no-flux tissue boundaries. The resulting framework is intended to support
mechanistic investigation of how chemokine-mediated immune recruitment and
chemotactic migration shape tumour--immune spatial organization.

\subsection{Biological setting: tumour--immune--chemokine interactions}
\label{subsec:biological-setting}

We consider a spatial tumour microenvironment represented by a bounded tissue
domain $\Omega\subset\mathbb R^d$, where $d=1,2$, or $3$ depending on the
application. The model contains three primary state variables: 
\begin{itemize}
    \item $T(t,x)$ is \text{tumour-cell density}.
    \item $E(t,x)$ is \text{immune effector-cell density}.
    \item $A(t,x)$ is \text{chemokine concentration}.
\end{itemize}

Here $t\ge 0$ denotes time and $x\in\Omega$ denotes spatial position. The
variable $T$ represents the local density of proliferating tumour cells. The
variable $E$ represents immune effector cells, such as cytotoxic lymphocytes or
other anti-tumour immune populations. The variable $A$ represents a chemokine or
aggregate chemotactic signal that promotes immune recruitment and guides immune
cell movement.

The biological feedback structure underlying the model is summarized in
Figure~\ref{fig:tumour-immune-chemokine-feedback}. The model is a computational oncology
framework coupling tumour growth, immune killing, chemokine production, immune
recruitment, and chemotactic immune migration.

\begin{figure}[htbp]
\centering
    \includegraphics[width=\linewidth]{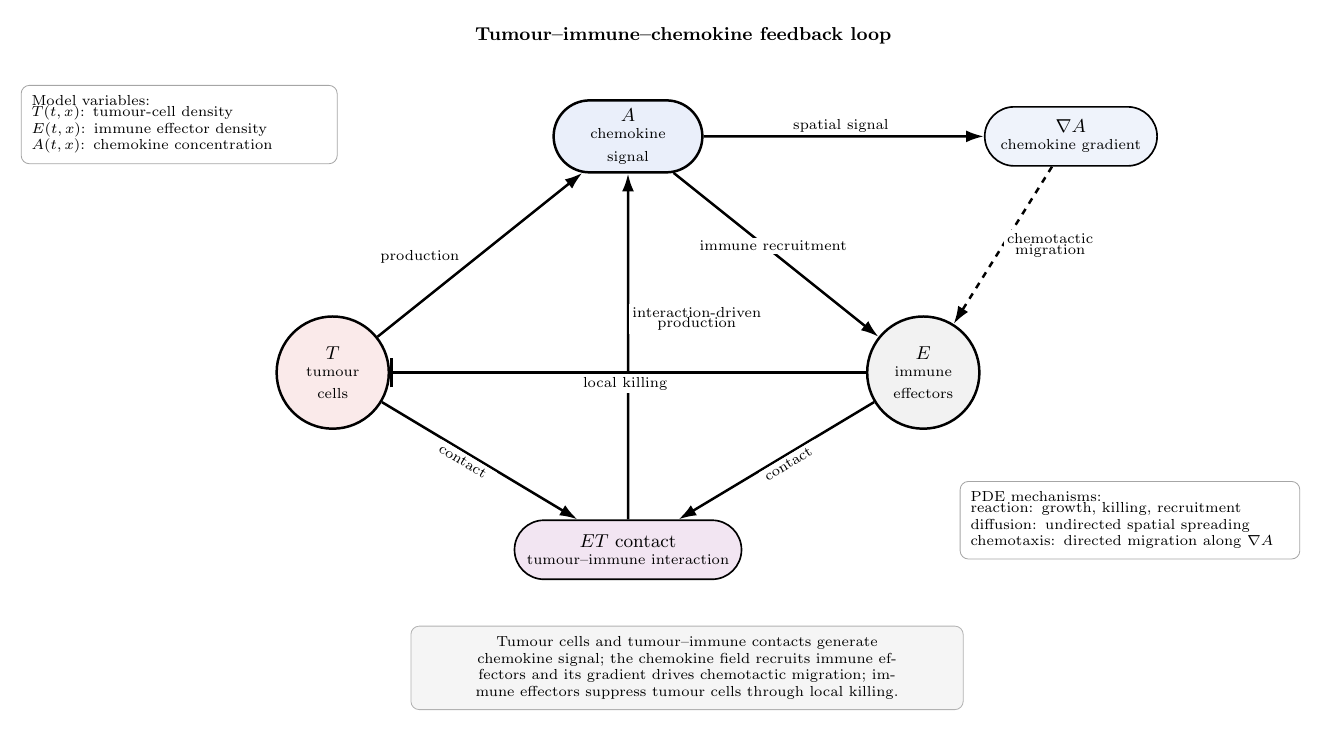}
\caption{
Tumour--immune--chemokine feedback loop used in the computational oncology
model. Tumour cells \(T\) produce chemokine \(A\), while tumour--immune contacts
\(ET\) provide an additional chemokine source. The chemokine field recruits
immune effector cells \(E\), and its spatial gradient \(\nabla A\) drives
chemotactic immune migration. Immune effectors suppress tumour cells through
local killing. This feedback structure is represented mathematically by the
reaction--diffusion--chemotaxis system \eqref{eq:dimensional-model}.}
\label{fig:tumour-immune-chemokine-feedback}
\end{figure}
\clearpage

The biological interpretation of the model is as follows. Tumour cells
proliferate locally and are limited by tissue-level carrying capacity, so the
baseline tumour dynamics are represented by logistic growth. Immune effector
cells kill tumour cells through local tumour--immune contact, producing a
negative interaction term in the tumour equation. Chemokines recruit immune
cells into the tumour microenvironment and create spatial gradients that direct
immune-cell migration. Thus, the immune population is influenced both by local
source terms, representing recruitment or activation, and by chemotactic drift
along the chemokine gradient. Finally, chemokines are produced by tumour cells
and by tumour--immune interactions, while natural degradation or clearance
reduces the chemokine concentration.

In schematic form, the biological mechanisms are
\[
T \ \text{grows logistically},
\qquad
E \ \text{kills } T,
\qquad
A \ \text{recruits and attracts } E,
\qquad
T \ \text{and } ET \ \text{produce } A .
\]
These assumptions lead naturally to a reaction--diffusion--chemotaxis system:
diffusion accounts for undirected spatial spreading, reaction terms describe
local tumour--immune--chemokine interactions, and chemotaxis represents directed
immune migration toward increasing chemokine concentration. The purpose of the
paper is to analyze this model mathematically, identify thresholds separating
homogeneous and spatially heterogeneous regimes, and construct a numerical
method that respects the biological structure of the equations.

\subsection{Related work and modelling context}
\label{subsec:related-work}

The present work sits at the intersection of several active research areas, which
we briefly survey here to situate our contribution.

\paragraph{Mathematical and computational oncology.}
Quantitative modelling has become central to the study of cancer as a multiscale,
spatially structured process. Integrative and computational frameworks combine
mechanistic models with experimental and clinical data to interrogate tumour
initiation, progression, heterogeneity, and treatment response
\cite{anderson2008integrative,altrock2015mathematics,yu2026microscopic,mcdonald2023computational,gopukumar2026multiscale,yu2026size}. Recent literature emphasizes both the
maturation of mechanistic models toward clinical decision support and the
increasing integration of mechanistic and data-driven approaches, including the
community roadmap of Rockne et al.\ \cite{rockne20192019} and the emerging programme
of mechanistic learning \cite{metzcar2024review,lan2025shallow,cai2026optimal,de2025radiation,laslo2025mechanistic}. The model studied here is mechanistic
and analytically tractable, in the spirit of this broader effort to extract
interpretable thresholds and mechanisms from minimal but biologically grounded
systems.

\paragraph{Tumour--immune interaction models.}
Dynamical models of tumour--immune competition originate with the work of
Kuznetsov et al.\ \cite{kuznetsov1994nonlinear,gao2022rolling}, whose ordinary-differential-equation
model of the cytotoxic T-lymphocyte response reproduces dormancy, immune escape,
and recurrent dynamics. Subsequent models incorporated immunotherapy and immune
recruitment \cite{kirschner1998modeling,de2001mathematical,mehdizadeh2023targeting,yu2026beyond,barrera2025impact}; a broad survey of
non-spatial tumour--immune models is given by Eftimie et al.\ \cite{eftimie2011interactions}.
These models capture the temporal balance between tumour growth and immune
control that our homogeneous equilibrium analysis also exhibits, but they do not
resolve spatial organization.

\paragraph{Spatial tumour modelling and the microenvironment.}
Spatially resolved models describe how local cell--signal interactions generate
tissue-scale structure. Reaction--diffusion and taxis-based models of
tumour-induced angiogenesis \cite{yu2026rigorous,hadjigeorgiou2024hybrid,anderson1998continuous,agosti2023image}, avascular tumour growth
\cite{wang2026breakdown,roose2007mathematical,ferreira2002reaction,chaplain1996avascular}, and the spatio-temporal response of cytotoxic T-lymphocytes to
a solid tumour \cite{matzavinos2004mathematical} demonstrate that directed
migration and diffusible signals can produce heterogeneous invasion and immune
infiltration patterns, including travelling waves and irregular spatial
distributions. Agent-based and hybrid platforms such as PhysiCell
\cite{ghaffarizadeh2018physicell,wang2025multi,savic2022heterogeneous} complement continuum approaches for the multicellular
tumour microenvironment. Furthermore, recent coupled PDE--agent-based frameworks have rigorously elucidated the bidirectional feedback between hypoxia-driven angiogenesis and therapeutic resistance, demonstrating how spatial vascular heterogeneity and resistant niches emerge from multiscale continuous and stochastic interactions \cite{wang2025analysis1}. Our model contributes a continuum treatment in which the
spatial structure arises specifically from chemokine-directed immune chemotaxis.

\paragraph{Chemokines, immune infiltration, and hot versus cold tumours.}
On the biological side, chemokines are the principal guidance cues that recruit
and position immune effector cells within the tumour microenvironment
\cite{nagarsheth2017chemokines,chaudhary2025role,liang2025global,foeng2022harnessing}. The density, location, and organization of tumour-infiltrating
lymphocytes are strongly prognostic \cite{galon2006type,lopez2025biological,yu2026pattern,brummel2023tumour}, motivating the now-standard
classification of tumours as immune ``hot'', ``altered'', or ``cold'' according
to their infiltration state \cite{galon2019approaches}. A central obstacle to
immunotherapy is immune exclusion, in which effector T cells are present but
physically kept away from cancer cells by features of the microenvironment
\cite{joyce2015t,yu2024extracellular,wang2026algebraic,bruni2023cancer}. The transition between spatially homogeneous coexistence
and chemotaxis-driven heterogeneity studied in this paper is a mathematical
caricature of precisely this hot/cold, infiltrated/excluded dichotomy.

\paragraph{Chemotaxis equations and their numerical analysis.}
Mathematically, the immune chemotaxis term places the model within the
Keller--Segel family of chemotaxis systems \cite{keller1970initiation,hillen2009user,wang2025analysis,horstmann20031970},
for which aggregation, blow-up, and boundedness under logistic damping or critical
sensitivity have been extensively analysed \cite{winkler2010aggregation,tao2012boundedness,wang2026damage,arumugam2021keller}.
Because chemotaxis systems concentrate mass and can form steep gradients,
specialized structure-preserving discretizations are required: conservative
finite-volume schemes \cite{filbet2006finite}, positivity
preserving central-upwind schemes \cite{chertock2008second,yu2026age}, and conservative
upwind finite-element methods \cite{saito2007conservative} have all been developed for
Keller--Segel-type models. Our finite-volume scheme with upwind chemotactic flux
and discrete positivity and residual diagnostics follows this line.

\paragraph{Contribution relative to the Special Issue.}
This paper contributes to advanced computational oncology in three ways. First,
it formulates a minimal tumour--immune--chemokine
reaction--diffusion--chemotaxis model that couples immune-mediated tumour
killing, chemokine-mediated immune recruitment, and directed immune migration
along chemokine gradients. Second, it derives analytically interpretable
thresholds for tumour-free control, homogeneous tumour--immune coexistence, and
chemotaxis-driven finite-wavelength instability, thereby connecting biological
feedback mechanisms with stability and pattern-formation criteria. Third, it
implements a conservative finite-volume scheme with no-flux boundary treatment,
upwind chemotactic flux, positivity monitoring, mass-balance diagnostics, and
post-hoc residual certification. These components link mathematical modelling,
PDE-based stability analysis, and numerically reliable simulation in a form that
is directly aligned with computational oncology studies of tumour-related
reaction--diffusion and chemotaxis systems.

\subsection{Outline}
\label{subsec:outline}

Section~\ref{sec:model-formulation} formulates the dimensional model, states the
boundary and initial conditions, and nondimensionalizes the system.
Section~\ref{sec:numerical-method} presents the structure-preserving
finite-volume scheme.
Section~\ref{sec:mathematical-analysis} establishes local solvability,
positivity, and a priori estimates; the technical proofs are collected in
Appendix~\ref{app:technical-estimates}.
Section~\ref{sec:equilibria-biological-regimes} identifies the equilibria and the
tumour-free immune-control threshold.
Section~\ref{sec:linear-stability} derives the dispersion relation about
coexistence, with the full Routh--Hurwitz algebra given in
Appendix~\ref{app:routh-hurwitz}.
Section~\ref{sec:numerical-experiments} reports the
numerical experiments, each verifying a specific theoretical result.
Section~\ref{sec:conclusion} concludes.

\section{Methods (Model Formulation)}\label{sec:model-formulation}

\subsection{Dimensional tumour--immune--chemokine model}
\label{subsec:dimensional-model}

Let $\Omega\subset\mathbb R^d$, $d\in\{1,2,3\}$, be a bounded smooth tissue
domain with outward unit normal vector $n$ on $\partial\Omega$. We denote by
$T(t,x)$, $E(t,x)$, and $A(t,x)$ tumour-cell density, immune effector-cell density,
and chemokine concentration, respectively, at time $t\ge0$ and position $x\in\Omega$.

We model tumour--immune--chemokine interactions by the following
reaction--diffusion--chemotaxis system:
\begin{equation}\label{eq:dimensional-model}
\begin{aligned}
T_t &= d_T\Delta T+rT\left(1-\frac{T}{K}\right)-\kappa ET,\\
E_t &= d_E\Delta E-\chi\nabla\cdot(E\nabla A)
      +s_0+s_1\frac{A}{\eta+A}-\mu E-\rho ET,\\
A_t &= d_A\Delta A+\alpha_TT+\alpha_CET-\lambda A.
\end{aligned}
\end{equation}
Here $d_T,d_E,d_A>0$ are diffusion coefficients. The tumour population grows
logistically with intrinsic growth rate $r>0$ and carrying capacity $K>0$.
The term $-\kappa ET$ represents immune-mediated tumour killing, with killing
rate $\kappa>0$. The immune population diffuses with coefficient $d_E>0$ and
moves chemotactically in response to the chemokine gradient $\nabla A$, with
chemotactic sensitivity $\chi\ge0$. The source term
\[
s_0+s_1\frac{A}{\eta+A}
\]
represents baseline immune influx together with chemokine-enhanced immune
recruitment. The parameter $s_0\ge0$ is the baseline immune supply, $s_1\ge0$ is
the maximal chemokine-induced recruitment rate, and $\eta>0$ is the half-saturation
constant. Immune cells are removed at rate $\mu>0$ and may be locally exhausted
or lost through tumour--immune interaction at rate $\rho>0$. The chemokine is
produced by tumour cells at rate $\alpha_T>0$ and by tumour--immune interactions
at rate $\alpha_C>0$, and is degraded or cleared at rate $\lambda>0$.
To visually summarize the complex spatial and biological interactions captured by the system \eqref{eq:dimensional-model}, we present a schematic diagram of the reaction--diffusion--chemotaxis topology in Figure \ref{fig:model_topology}. 

\begin{figure}[htbp]
    \centering
    \includegraphics[width=\linewidth]{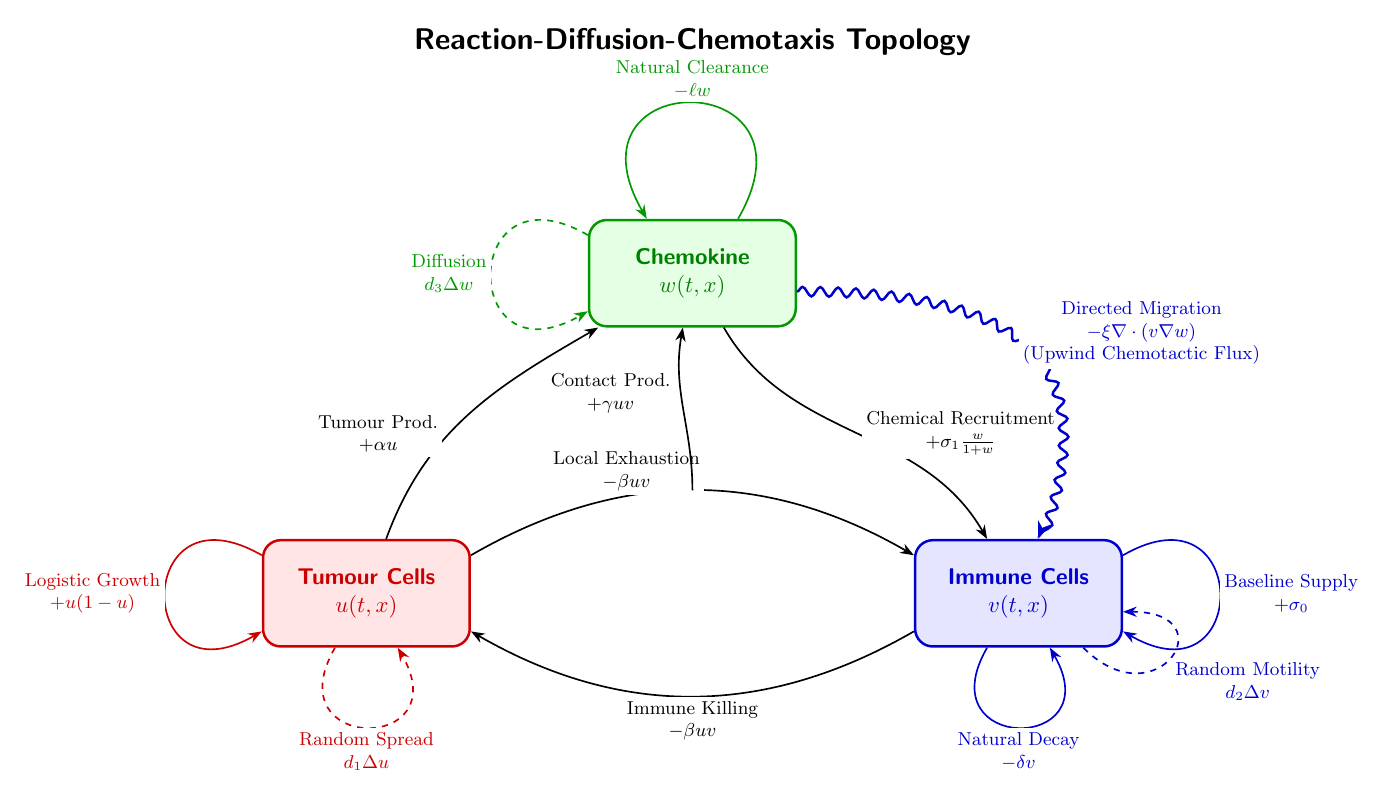}
    \caption{Schematic representation of the reaction--diffusion--chemotaxis topology for the tumour--immune--chemokine interactions. Solid lines indicate local reaction kinetics, such as population growth, cell death, and chemical secretion. Dashed lines represent random spatial motility (diffusion), while the wavy blue line illustrates the directed chemotactic migration of immune cells up the chemokine gradient. The non-dimensional variables $(u, v, w)$ depicted here correspond to the populations $(T, E, A)$.}
    \label{fig:model_topology}
\end{figure}
\clearpage

The model is intentionally minimal. It retains the mechanisms needed to describe
spatial tumour growth, immune killing, chemokine-mediated immune recruitment,
and directed immune migration, while remaining analytically tractable and
numerically testable. In particular, the chemotaxis term
$\displaystyle -\chi\nabla\cdot(E\nabla A)$
is the mechanism through which chemokine gradients can amplify spatial
heterogeneity and generate immune infiltration or immune-exclusion patterns.

\subsection{Boundary and initial conditions}
\label{subsec:boundary-initial-conditions}

We impose biologically natural no-flux boundary conditions. These conditions
describe a closed tissue region across whose boundary there is no net movement
of tumour cells, immune cells, or chemokine:
\begin{equation}\label{eq:noflux}
\nabla T\cdot n=0,\qquad
(d_E\nabla E-\chi E\nabla A)\cdot n=0,\qquad
\nabla A\cdot n=0
\quad\text{on }\partial\Omega,\ t>0 .
\end{equation}
The immune-cell condition is written in conservative form because the total
immune flux is
\[
J_E=-d_E\nabla E+\chi E\nabla A .
\]
Thus, $J_E\cdot n=0$ is equivalent to
\[
(d_E\nabla E-\chi E\nabla A)\cdot n=0.
\]
This is the physical no-flux condition for the chemotactic immune population.

The initial data are assumed to satisfy
\begin{equation}\label{eq:initial-data-dimensional}
T(0,x)=T_0(x),\qquad
E(0,x)=E_0(x),\qquad
A(0,x)=A_0(x),
\qquad x\in\Omega,
\end{equation}
where
\[
T_0(x)\ge0,\qquad E_0(x)\ge0,\qquad A_0(x)\ge0
\quad\text{for all }x\in\Omega .
\]
When classical solutions are considered, the initial data are also assumed to be
sufficiently smooth and compatible with the boundary conditions.

For the eigenfunction-based linear stability calculation in
Section~\ref{sec:linear-stability}, we impose the stronger componentwise
Neumann condition
\[
\nabla T\cdot n=\nabla E\cdot n=\nabla A\cdot n=0
\quad\text{on }\partial\Omega .
\]
This condition implies the natural no-flux condition \eqref{eq:noflux} whenever
$\nabla A\cdot n=0$. Its purpose is analytical: it allows the Neumann Laplacian
eigenbasis to be used directly in the linear stability and dispersion-relation
calculations. Table~\ref{tab:dimensional-parameters} lists dimensional parameters in the tumour--immune--chemokine model.

\begin{table}[htbp]
\centering
\caption{Dimensional parameters with representative biological orders of
magnitude. Ranges are indicative of solid-tumour tissue; cell motilities are
effective (random-motility) coefficients from intravital imaging, chemokine
transport from tissue diffusion measurements, and kinetic rates from validated
tumour--immune models. Values are used only to fix orders of magnitude, not to
model a specific tumour.}
\label{tab:dimensional-parameters}
\begin{tabularx}{\textwidth}{@{} l >{\raggedright\arraybackslash}X >{\raggedright\arraybackslash}X l @{}}
\toprule
Symbol & Meaning & Representative estimate & Source \\
\midrule
$d_T$ & tumour-cell random motility & $\sim 10^{-2}$--$1\ \mu\mathrm{m}^2\,\mathrm{s}^{-1}$ & (weakly motile) \\
$d_E$ & immune-cell motility        & $\sim 1$--$70\ \mu\mathrm{m}^2\,\mathrm{min}^{-1}$ ($\sim0.02$--$1\ \mu\mathrm{m}^2\,\mathrm{s}^{-1}$) & \cite{miller2003autonomous} \\
$d_A$ & chemokine diffusion         & $\sim 10$--$100\ \mu\mathrm{m}^2\,\mathrm{s}^{-1}$ ($10^{-7}$--$10^{-6}\ \mathrm{cm}^2\,\mathrm{s}^{-1}$) & \cite{thurley2015three} \\
$r$   & tumour growth rate          & $\sim 0.1$--$0.5\ \mathrm{day}^{-1}$ (doubling days--weeks) & \cite{de2005validated} \\
$K$   & tumour carrying capacity    & $\sim 10^{8}$--$10^{9}\ \mathrm{cells\,cm}^{-3}$ & tissue density \\
$\kappa$ & immune killing rate      & lysis-based, $O(10^{-1})\ \mathrm{day}^{-1}$ per effector density & \cite{de2005validated} \\
$\chi$ & chemotactic sensitivity    & $\sim 10^{-8}$--$10^{-6}\ \mathrm{cm}^2\,\mathrm{s}^{-1}$ per unit signal & (chemotaxis assays) \\
$s_0$ & baseline immune influx      & $\sim 10^{-2}\ \mathrm{day}^{-1}$ & \cite{de2005validated} \\
$s_1$ & max.\ chemokine recruitment & comparable to/above $s_0$ & \cite{de2005validated} \\
$\eta$ & half-saturation constant   & signal-dependent (nM range) & (receptor affinity) \\
$\mu$ & immune-cell loss rate       & $\sim 0.01$--$0.2\ \mathrm{day}^{-1}$ & \cite{de2005validated} \\
$\rho$ & tumour-induced immune loss  & $O(\mu)$ & \cite{de2005validated} \\
$\alpha_T,\alpha_C$ & chemokine production & per-cell secretion, $O(10^{-4})\ \mathrm{ng}\,\mathrm{cells}^{-1}\mathrm{day}^{-1}$ & \cite{poznansky2000active,sutherland1990functional} \\
$\lambda$ & chemokine degradation   & $\sim 0.03$--$1\ \mathrm{h}^{-1}$ (half-life min--h) & \cite{thurley2015three} \\
$L$ & reference tissue length       & $\sim 10^{2}\ \mu\mathrm{m}$--$1\ \mathrm{cm}$ & (signalling niche) \\
\bottomrule
\end{tabularx}
\end{table}

\subsection{Nondimensionalization}
\label{subsec:nondimensionalization}

We now nondimensionalize \eqref{eq:dimensional-model}. Let
\[
\widetilde t=rt,\qquad
\widetilde x=\frac{x}{L},\qquad
u=\frac{T}{K},\qquad
v=\frac{\kappa}{r}E,\qquad
w=\frac{A}{\eta},
\]
where $L>0$ is a reference tissue length scale. In these variables, $u$ is the
tumour density normalized by the carrying capacity, $v$ is the immune density
scaled by the tumour-killing rate relative to tumour growth, and $w$ is the
chemokine concentration normalized by the immune-recruitment half-saturation
constant.

Dropping tildes after rescaling, we obtain
\begin{equation}\label{eq:nondim-model}
\begin{aligned}
u_t &= d_1\Delta u+u(1-u-v),\\
v_t &= d_2\Delta v-\xi\nabla\cdot(v\nabla w)
      +\sigma_0+\sigma_1\frac{w}{1+w}-\delta v-\beta uv,\\
w_t &= d_3\Delta w+\alpha u+\gamma uv-\ell w.
\end{aligned}
\end{equation}
The nondimensional diffusion coefficients are
\[
d_1=\frac{d_T}{rL^2},\qquad
d_2=\frac{d_E}{rL^2},\qquad
d_3=\frac{d_A}{rL^2}.
\]
The nondimensional chemotactic sensitivity is
\[
\xi=\frac{\chi\eta}{rL^2}.
\]
The immune recruitment, decay, and tumour--immune loss parameters are
\[
\sigma_0=\frac{\kappa s_0}{r^2},\qquad
\sigma_1=\frac{\kappa s_1}{r^2},\qquad
\delta=\frac{\mu}{r},\qquad
\beta=\frac{\rho K}{r}.
\]
The chemokine production and degradation parameters are
\[
\alpha=\frac{\alpha_T K}{r\eta},\qquad
\gamma=\frac{\alpha_C K}{\kappa\eta},\qquad
\ell=\frac{\lambda}{r}.
\]
The nondimensional no-flux boundary conditions become
\begin{equation}\label{eq:nondim-noflux}
\nabla u\cdot n=0,\qquad
(d_2\nabla v-\xi v\nabla w)\cdot n=0,\qquad
\nabla w\cdot n=0
\quad\text{on }\partial\Omega,\ t>0 .
\end{equation}
For the analytical linear stability calculation, we again use the stronger
componentwise Neumann condition
\begin{equation}\label{eq:no_flux_pointwise}
\nabla u\cdot n=\nabla v\cdot n=\nabla w\cdot n=0
\quad\text{on }\partial\Omega .
\end{equation}

Table~\ref{tab:nondimensional-parameters} lists nondimensional parameters. The nondimensional system \eqref{eq:nondim-model} is the main model analyzed in
the remainder of the paper. The first equation describes logistic tumour growth
and immune-mediated tumour removal. The second equation describes immune-cell
diffusion, chemotactic migration, chemokine-mediated recruitment, natural
immune-cell loss, and tumour-induced immune loss. The third equation describes
chemokine diffusion, production by tumour cells and tumour--immune interactions,
and chemokine degradation. This formulation makes explicit the mechanisms by
which chemokine-mediated immune recruitment and chemotactic migration may
stabilize a homogeneous tumour--immune coexistence state or destabilize it
toward spatially heterogeneous patterns.

\begin{table}[htbp]
\centering
\caption{Nondimensional parameters in \eqref{eq:nondim-model}.}
\label{tab:nondimensional-parameters}
\begin{tabularx}{\textwidth}{lll}
\toprule
Symbol & Definition & Interpretation \\
\midrule
$d_1$ & $d_T/(rL^2)$ & tumour diffusion relative to growth \\
$d_2$ & $d_E/(rL^2)$ & immune diffusion relative to tumour growth \\
$d_3$ & $d_A/(rL^2)$ & chemokine diffusion relative to tumour growth \\
$\xi$ & $\chi\eta/(rL^2)$ & chemotactic sensitivity \\
$\sigma_0$ & $\kappa s_0/r^2$ & baseline immune supply \\
$\sigma_1$ & $\kappa s_1/r^2$ & chemokine-induced immune recruitment \\
$\delta$ & $\mu/r$ & immune decay relative to tumour growth \\
$\beta$ & $\rho K/r$ & tumour-induced immune loss strength \\
$\alpha$ & $\alpha_TK/(r\eta)$ & tumour-driven chemokine production \\
$\gamma$ & $\alpha_CK/(\kappa\eta)$ & interaction-driven chemokine production \\
$\ell$ & $\lambda/r$ & chemokine degradation relative to tumour growth \\
\bottomrule
\end{tabularx}
\end{table}

\paragraph{Biological plausibility of the nondimensional base set.}
The nondimensional groups inherit clear biological orderings. Because chemokine molecules diffuse far faster than immune cells migrate, which in turn move far faster than bulk tumour cells, one has $d_A \gg d_E \gg d_T$ and hence $d_3 \gg d_2 \gg d_1$; the base values $d_3=1$, $d_2=0.1$, $d_1=0.01$ reproduce exactly this three-order separation. A representative dimensional set ($r=0.2\ \mathrm{day}^{-1}$, tissue scale $L\approx 6$--$7\ \mathrm{mm}$, motilities $d_T:d_E:d_A \approx 0.1:1:10\ \mu\mathrm{m}^2\,\mathrm{s}^{-1}$, immune loss $\mu=0.06\ \mathrm{day}^{-1}$) maps onto the base values $d_1=0.01$, $d_2=0.10$, $d_3=1.00$, $\delta=0.30$, $\ell=1.00$. We note one deliberate simplification: fast chemokine turnover ($\lambda\sim 1\ \mathrm{h}^{-1}$) would give $\ell\gg 1$, placing the chemokine near a quasi-steady state; we retain $\ell=O(1)$ for analytical transparency, which rescales but does not remove the chemotaxis-driven finite-wavelength instability.

\section{Methods (Numerical Scheme)}
\label{sec:numerical-method}

This section describes the numerical method used to simulate the nondimensional tumour--immune--chemokine system \eqref{eq:nondim-model}.
The numerical method is designed around three requirements: 
(i) $u_h$, $v_h$, $w_h\ge0$, (ii) discrete no-flux boundary conditions,
(iii) mass-balance and residual diagnostics.
These requirements are important because negative cell densities or artificial
boundary fluxes would have no biological interpretation and could distort the
onset of spatial tumour--immune patterns. Chemotaxis systems are well known to
concentrate mass and form steep gradients, which has motivated a substantial
literature on structure-preserving discretizations for Keller--Segel-type models
\cite{filbet2006finite,chertock2008second,yu2026fullcovariance,saito2007conservative}.

For clarity, we first describe the method on a one-dimensional interval
\(\Omega=(0,L)\). The extension to rectangular two-dimensional domains is
obtained by applying the same finite-volume construction in each coordinate
direction.

\subsection{Finite-volume discretization and no-flux treatment}
\label{subsec:finite-volume-discretization}

Let \(N\in\mathbb N\), \(h=L/N\), and define cell centers
\[
x_i=\left(i-\frac12\right)h,\qquad i=1,\dots,N.
\]
Let \(u_i(t)\), \(v_i(t)\), and \(w_i(t)\) denote cell averages over the finite
volume
\[
C_i=(x_{i-1/2},x_{i+1/2}).
\]
The semi-discrete finite-volume scheme is written in conservative form as
\begin{equation}\label{eq:fv-semidiscrete-u}
\frac{d u_i}{dt}
=
-\frac{F^u_{i+1/2}-F^u_{i-1/2}}{h}
+
u_i(1-u_i-v_i),
\end{equation}
\begin{equation}\label{eq:fv-semidiscrete-v}
\frac{d v_i}{dt}
=
-\frac{F^v_{i+1/2}-F^v_{i-1/2}}{h}
+
\sigma_0+\sigma_1\frac{w_i}{1+w_i}
-\delta v_i-\beta u_i v_i,
\end{equation}
and
\begin{equation}\label{eq:fv-semidiscrete-w}
\frac{d w_i}{dt}
=
-\frac{F^w_{i+1/2}-F^w_{i-1/2}}{h}
+
\alpha u_i+\gamma u_i v_i-\ell w_i .
\end{equation}

The tumour and chemokine diffusive fluxes are approximated by centered finite
differences:
\begin{equation}\label{eq:u-w-fluxes}
F^u_{i+1/2}
=
-d_1\frac{u_{i+1}-u_i}{h},
\qquad
F^w_{i+1/2}
=
-d_3\frac{w_{i+1}-w_i}{h}.
\end{equation}
The immune-cell flux contains both diffusion and chemotaxis. We write the
continuous immune flux as
\[
J_v=-d_2\nabla v+\xi v\nabla w .
\]
Accordingly, the numerical immune flux is
\begin{equation}\label{eq:v-flux}
F^v_{i+1/2}
=
-d_2\frac{v_{i+1}-v_i}{h}
+
\xi v^{\rm up}_{i+1/2}\frac{w_{i+1}-w_i}{h}.
\end{equation}
Here \(v^{\rm up}_{i+1/2}\) is the upwind value selected according to the
discrete chemotactic velocity
\[
a_{i+1/2}
=
\xi\frac{w_{i+1}-w_i}{h}.
\]
Specifically,
\begin{equation}\label{eq:upwind-value}
v^{\rm up}_{i+1/2}
=
\begin{cases}
v_i, & a_{i+1/2}\ge0,\\[2mm]
v_{i+1}, & a_{i+1/2}<0.
\end{cases}
\end{equation}
The chemotactic flux is therefore discretized by an upwind finite-volume flux.
This choice is first-order in the advective chemotaxis term but improves
robustness and positivity in parameter regimes where steep immune gradients
form.

The no-flux boundary conditions are imposed by setting the boundary fluxes to
zero:
\begin{equation}\label{eq:discrete-noflux}
F^u_{1/2}=F^u_{N+1/2}=0,\qquad
F^v_{1/2}=F^v_{N+1/2}=0,\qquad
F^w_{1/2}=F^w_{N+1/2}=0.
\end{equation}
This implementation is conservative at the discrete level. Summing
\eqref{eq:fv-semidiscrete-u}--\eqref{eq:fv-semidiscrete-w} over all cells gives
cancellation of interior fluxes and leaves only reaction contributions.
The finite-volume spatial discretization framework is illustrated in Figure \ref{fig:fv_scheme}.

\begin{figure}[htbp]
    \centering
    \includegraphics[width=\linewidth]{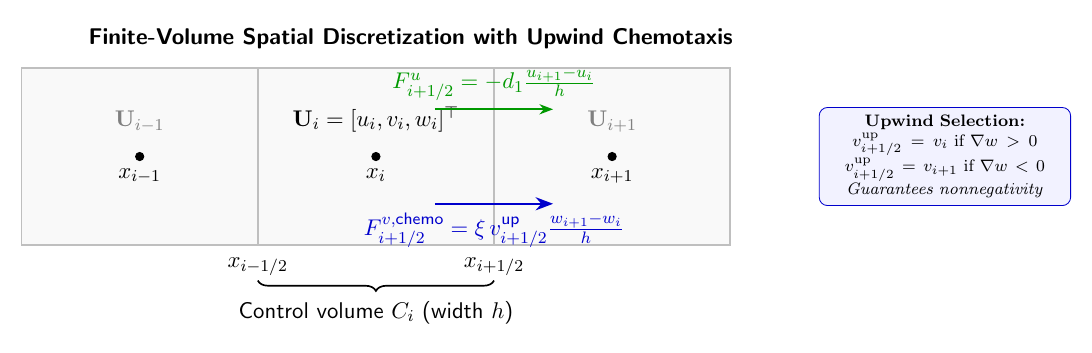}
    \caption{Schematic of the 1D finite-volume spatial discretization framework. The domain is divided into control volumes $C_i$ of uniform width $h$, with nodes located at the cell centers $x_i$ and boundaries at the interfaces $x_{i\pm1/2}$. A centered difference approximation evaluates the diffusion flux $F^u_{i+1/2}$. In contrast, an upwind scheme dictates the chemotactic flux $F^{v,\text{chemo}}_{i+1/2}$ to effectively capture directed migration and ensure the strict non-negativity of cell densities.}
    \label{fig:fv_scheme}
\end{figure}
\clearpage

\subsection{Mass-balance diagnostics}
\label{subsec:mass-balance-diagnostics}

The discrete masses are defined by
\[
M_u(t)=h\sum_{i=1}^N u_i(t),\qquad
M_v(t)=h\sum_{i=1}^N v_i(t),\qquad
M_w(t)=h\sum_{i=1}^N w_i(t).
\]
Because of the conservative finite-volume form and the no-flux boundary
treatment, the numerical masses satisfy the discrete balance laws
\begin{equation}\label{eq:discrete-mass-u}
\frac{dM_u}{dt}
=
h\sum_{i=1}^N u_i(1-u_i-v_i),
\end{equation}
\begin{equation}\label{eq:discrete-mass-v}
\frac{dM_v}{dt}
=
h\sum_{i=1}^N
\left(
\sigma_0+\sigma_1\frac{w_i}{1+w_i}
-\delta v_i-\beta u_i v_i
\right),
\end{equation}
and
\begin{equation}\label{eq:discrete-mass-w}
\frac{dM_w}{dt}
=
h\sum_{i=1}^N
\left(
\alpha u_i+\gamma u_i v_i-\ell w_i
\right).
\end{equation}
These identities are used as mass-balance diagnostics in the numerical
experiments. In particular, the discrepancy between the left-hand side and the
right-hand side of \eqref{eq:discrete-mass-u}--\eqref{eq:discrete-mass-w}
provides a simple check on temporal accuracy and boundary-flux implementation.

\subsection{BDF time stepping with positivity diagnostics}
\label{subsec:bdf-time-stepping}

The semi-discrete system
\eqref{eq:fv-semidiscrete-u}--\eqref{eq:fv-semidiscrete-w} is stiff when
diffusion is strong, the mesh is fine, or chemotactic gradients are large.
Therefore, we use an implicit backward differentiation formula (BDF) time
discretization. Let
\[
U^n=(u_1^n,\dots,u_N^n,v_1^n,\dots,v_N^n,w_1^n,\dots,w_N^n)^\top
\]
denote the vector of all cell averages at time \(t_n\). In compact form, the
semi-discrete system is
\[
\frac{dU}{dt}=\mathcal R_h(U),
\]
where \(\mathcal R_h\) contains the finite-volume fluxes and reaction terms.

For example, the first-order BDF scheme is backward Euler:
\begin{equation}\label{eq:backward-euler}
\frac{U^{n+1}-U^n}{\Delta t}
=
\mathcal R_h(U^{n+1}).
\end{equation}
In computations requiring higher temporal accuracy, a second-order BDF method is
used after one backward Euler starting step:
\begin{equation}\label{eq:bdf2}
\frac{3U^{n+1}-4U^n+U^{n-1}}{2\Delta t}
=
\mathcal R_h(U^{n+1}).
\end{equation}
The nonlinear algebraic system at each time step is solved by Newton iteration
with tolerance \(\varepsilon_{\rm newton}\). Time steps are reduced if Newton
iteration fails to converge or if the positivity diagnostic described below is
violated beyond the prescribed tolerance.

We do not claim that the BDF method by itself guarantees positivity. Instead,
positivity is promoted by the conservative finite-volume form, the no-flux
boundary treatment, and the upwind chemotactic flux, and is monitored explicitly
during the computation. At each time step we compute
\begin{equation}\label{eq:positivity-diagnostic}
m_{\min}^{n}
=
\min\left\{
\min_i u_i^n,\,
\min_i v_i^n,\,
\min_i w_i^n
\right\}.
\end{equation}
The run is accepted only if
\begin{equation}\label{eq:positivity-tolerance}
m_{\min}^{n}\ge -\varepsilon_{\rm pos}
\end{equation}
for all reported time steps, where \(\varepsilon_{\rm pos}\) is a small solver
tolerance. In all reported runs, the minimum of each component remains
nonnegative up to solver tolerance.

\subsection{Residual diagnostics}
\label{subsec:residual-diagnostics}

To distinguish genuine pattern formation from numerical artifacts, we compute
post-hoc residuals for the three equations. Let \(D_t\), \(\Delta_h\), and
\(\nabla_h\cdot\) denote the discrete time derivative, Laplacian, and divergence
operators induced by the finite-volume method. The residuals are defined by
\begin{equation}\label{eq:residual-u}
R_u^n
=
D_t u^n
-
d_1\Delta_h u^n
-
u^n(1-u^n-v^n),
\end{equation}
\begin{equation}\label{eq:residual-v}
R_v^n
=
D_t v^n
-
d_2\Delta_h v^n
+
\xi\nabla_h\cdot(v^n\nabla_h w^n)
-
\sigma_0-\sigma_1\frac{w^n}{1+w^n}
+\delta v^n+\beta u^n v^n,
\end{equation}
and
\begin{equation}\label{eq:residual-w}
R_w^n
=
D_t w^n
-
d_3\Delta_h w^n
-
\alpha u^n-\gamma u^n v^n+\ell w^n.
\end{equation}
We report the normalized residual indicator
\begin{equation}\label{eq:normalized-residual}
\mathcal E_{\rm res}^n
=
\frac{
\|R_u^n\|_2+\|R_v^n\|_2+\|R_w^n\|_2
}{
1+\|u^n\|_2+\|v^n\|_2+\|w^n\|_2
}.
\end{equation}
Together with the positivity diagnostic and the mass-balance identities,
\(\mathcal E_{\rm res}^n\) provides a numerical certificate that the computed
spatial patterns are consistent with the PDE system and not merely artifacts of
the discretization.

\subsection{Consistency, positivity, and convergence status of the scheme}
\label{subsec:positivity-cfl}

We complement the a posteriori diagnostics with an a priori analysis of the
finite-volume discretization.

\medskip\noindent\textbf{Consistency.}
On a smooth solution, Taylor expansion of the fluxes about each face gives, for
the diffusion and reaction terms, a local truncation error of order
$\mathcal O(h^2)$ (centred face gradients and midpoint reaction quadrature are
second order). The first-order upwind chemotactic flux
\eqref{eq:v-flux}--\eqref{eq:upwind-value} contributes a numerical-diffusion
error of size
$\dfrac h2\,|a_{i+1/2}|\,\partial_{xx}v(x_i)+\dfrac{h}{2}\partial_x(|a|(x_i))\partial_xv(x_{i-1/2})+\mathcal O(h^2)$, i.e.\
$\mathcal O(h)$ where the chemotactic drift $a=\xi\partial_x w$ is nonzero. Hence
the semi-discrete scheme is consistent, with local truncation error
\[
\begin{aligned}
\tau_h&=\mathcal O(h)\ \text{in general},\\[3pt]
\tau_h&=\mathcal O(h^2)\ \text{where the leading upwind defect $\partial_x(\lvert a\rvert \partial_xv)$ vanishes.}
\end{aligned}
\]
This is consistent with the observed order $1.92$ in
Experiment~\ref{subsec:exp6-convergence}, where the smooth diffusion--reaction
part dominates; the order degrades toward $1$ only in strongly
chemotaxis-dominated regimes with steep immune gradients.

\medskip\noindent\textbf{Positivity with a precise step-size condition.}
We \emph{prove} that the explicit scheme preserves nonnegativity under a
computable step-size restriction. Consider the fully explicit forward-Euler
update of the finite-volume scheme on a uniform Cartesian mesh of spacing $h$ in
dimension $D\in\{1,2,3\}$, with no-flux boundaries imposed as zero boundary-face
fluxes. For an interior face in coordinate direction $e$ let
$a_{i+\frac12 e}=\xi\,(w_{i+e}-w_i)/h$ be the discrete chemotactic velocity, and
write $a^{\pm}=\max\{\pm a,0\}$. Let $n_i\le 2D$ be the number of present
neighbours of cell $i$, and set
\[
M_u^n=\max_i u_i^n,\qquad M_v^n=\max_i v_i^n,\qquad
G^n=\max_{\text{faces}}\frac{|w^n_{j}-w^n_i|}{h}.
\]

\begin{proposition}[discrete positivity]\label{prop:positivity-cfl}
Assume $u^n,v^n,w^n\ge0$ componentwise. If
\begin{equation}\label{eq:cfl-positivity}
\Delta t\left(\frac{2D\max\{d_1,d_2,d_3\}}{h^2}
+\frac{2D\,\xi\,G^n}{h}
+\max\{\,M_u^n+M_v^n,\ \delta+\beta M_u^n,\ \ell\,\}\right)\le 1,
\end{equation}
then $u^{n+1},v^{n+1},w^{n+1}\ge0$ componentwise. Hence if
\eqref{eq:cfl-positivity} holds at every step the scheme preserves nonnegativity
for all $n$.    
\end{proposition}

\begin{proof}
The upwind rule \eqref{eq:upwind-value} writes the discrete chemotactic flux as
$a_{i+\frac12 e}\,v^{\rm up}_{i+\frac12 e}
=a^{+}_{i+\frac12 e}v_i-a^{-}_{i+\frac12 e}v_{i+e}$. Substituting into the update
and collecting terms, each component takes the form
\[
c_i^{\,n+1}=\sum_{j\sim i}\theta_{ij}\,c_j^{\,n}+\Delta t\,S_i
+\bigl(1-\Delta t\,\kappa_i\bigr)c_i^{\,n},
\qquad \theta_{ij}\ge0,\ \ S_i\ge0,
\]
as follows.

\emph{Tumour.}
$u_i^{n+1}=\dfrac{\Delta t\,d_1}{h^2}\sum_{j\sim i}u_j^{\,n}
+\bigl(1-\Delta t\,\kappa_i^{u}\bigr)u_i^{\,n}$, with
$\kappa_i^{u}=\dfrac{d_1 n_i}{h^2}-\bigl(1-u_i^{n}-v_i^{n}\bigr)
\le \dfrac{2D d_1}{h^2}+M_u^n+M_v^n$
(the growth $+u_i$ is retained inside the centre weight and only lowers
$\kappa_i^{u}$).

\emph{Chemokine.}
$w_i^{n+1}=\dfrac{\Delta t\,d_3}{h^2}\sum_{j\sim i}w_j^{\,n}
+\Delta t\bigl(\alpha u_i^{n}+\gamma u_i^{n}v_i^{n}\bigr)
+\bigl(1-\Delta t\,\kappa_i^{w}\bigr)w_i^{\,n}$, with
$\kappa_i^{w}=\dfrac{d_3 n_i}{h^2}+\ell\le \dfrac{2D d_3}{h^2}+\ell$ and source
$\ge0$.

\emph{Immune.}
The neighbour weight of $v_{i-e}$ is
$\Delta t\bigl(\dfrac{d_2}{h^2}+\dfrac{a^{+}_{i-\frac12 e}}{h}\bigr)\ge0$ and of
$v_{i+e}$ is
$\Delta t\bigl(\dfrac{d_2}{h^2}+\dfrac{a^{-}_{i+\frac12 e}}{h}\bigr)\ge0$; the
source $\Delta t\bigl(\sigma_0+\sigma_1\dfrac{w_i}{1+w_i}\bigr)\ge0$; and
\[
\kappa_i^{v}=\frac{d_2 n_i}{h^2}
+\frac1h\sum_{e}\bigl(a^{+}_{i+\frac12 e}+a^{-}_{i-\frac12 e}\bigr)
+\delta+\beta u_i^{n}
\le \frac{2D d_2}{h^2}+\frac{2D\,\xi G^n}{h}+\delta+\beta M_u^n,
\]
where $\sum_{e}\bigl(a^{+}_{i+\frac12 e}+a^{-}_{i-\frac12 e}\bigr)
\le\sum_{\text{faces at }i}|a|\le 2D\,\xi G^n$.

In every case the neighbour weights and sources are nonnegative and the centre
weight is $1-\Delta t\,\kappa_i$. The bracket in \eqref{eq:cfl-positivity}
dominates each of $\kappa_i^{u},\kappa_i^{v},\kappa_i^{w}$, so
\eqref{eq:cfl-positivity} gives $\Delta t\,\kappa_i\le1$ for every cell and
component; hence every centre weight is nonnegative. Thus $c_i^{\,n+1}$ is a
nonnegative combination of nonnegative quantities, so $c_i^{\,n+1}\ge0$.
Induction on $n$ completes the proof.
\end{proof}

\noindent Because $v$ has no a priori $L^\infty$ bound (only the mass bound of
Proposition~\ref{prop:mass-estimates}), the peak $M_v^n$ may grow as immune cells
aggregate, and \eqref{eq:cfl-positivity} then automatically reduces $\Delta t$ ---
consistent with Remark~\ref{rem:no-global-linfty-overclaim}; the tumour factor is
controlled by the discrete logistic bound
$M_u^n\le\max\{1,\|u^0\|_\infty\}$. In the implementation the step is selected
adaptively to satisfy \eqref{eq:cfl-positivity} with a safety factor below one, so
positivity holds by the Proposition. The implicit BDF integrator used for the
stiff one-dimensional runs carries no such guarantee and is instead monitored a
posteriori via \eqref{eq:positivity-diagnostic}--\eqref{eq:positivity-tolerance}.

\medskip\noindent\textbf{Convergence status.}
For sufficiently smooth solutions, the spatial finite-volume operator is
consistent, with truncation error \(O(h)\) in general because of the upwind
chemotactic flux and \(O(h^2)\) when its leading first-order contribution
vanishes. Proposition~1 establishes nonnegativity of the forward-Euler
discretization under~\eqref{eq:cfl-positivity}, while the conservative flux form gives the
discrete mass-balance identities. These properties alone, however, do not
constitute a convergence proof for the nonlinear cross-diffusion system.
Such a proof would additionally require mesh-uniform stability estimates and
either a discrete compactness argument or a direct error estimate around the
classical solution; see, for example,
\cite{eymard2000finite,filbet2006finite,liu2026ftu}. We therefore interpret the
grid-refinement results of Experiment~7.6 as numerical evidence of
self-convergence. The observed order \(1.92\) reflects the smooth,
diffusion-dominated regime considered there; the formal asymptotic spatial
order of the upwind scheme is first order in general.

\subsection{Extension to two-dimensional rectangular domains}
\label{subsec:two-dimensional-extension}

On a rectangular domain
\[
\Omega=(0,L_x)\times(0,L_y),
\]
we use a Cartesian finite-volume grid with cell averages
\[
u_{i,j}(t),\qquad v_{i,j}(t),\qquad w_{i,j}(t).
\]
Diffusive fluxes are computed on cell faces by centered differences, while the
chemotactic immune flux is computed face-by-face using the upwind value of
\(v\) determined by the corresponding discrete chemotactic velocity. For
example, in the \(x\)-direction,
\[
F^{v,x}_{i+1/2,j}
=
-d_2\frac{v_{i+1,j}-v_{i,j}}{h_x}
+
\xi v^{\rm up}_{i+1/2,j}
\frac{w_{i+1,j}-w_{i,j}}{h_x},
\]
and the analogous formula is used in the \(y\)-direction. No-flux boundary
conditions are again imposed by setting all boundary face fluxes to zero.

This dimension-by-dimension finite-volume construction preserves the same
conservative structure as the one-dimensional scheme. The positivity,
mass-balance, and residual diagnostics are computed in the same way, with sums
taken over all grid cells.

\section{Results (Mathematical Well-posedness)}\label{sec:mathematical-analysis}

In this section, we establish the basic analytical properties of the
nondimensional tumour--immune--chemokine system \eqref{eq:nondim-model}
posed in a bounded smooth domain $\Omega\subset\mathbb R^d$ with nonnegative
initial data. The purpose of this section is to identify the mathematical
regime in which the model is biologically meaningful: solutions exist at least
locally, remain nonnegative, and satisfy basic a priori estimates. These
estimates also provide the analytical foundation for the equilibrium and linear
stability analysis in the following sections. The technical proofs are collected
in Appendix~\ref{app:technical-estimates}.

Throughout the paper we assume
\[
d_1,d_2,d_3>0,\qquad
\xi,\sigma_0,\sigma_1,\delta,\beta,\alpha,\gamma,\ell\ge 0,
\qquad
\delta>0,\quad \ell>0.
\]
Unless otherwise stated, solutions are considered under the natural no-flux
boundary conditions \eqref{eq:nondim-noflux}.
For the eigenfunction-based linear stability calculation, we shall use the
stronger componentwise Neumann condition \eqref{eq:no_flux_pointwise}
as already explained in Section~\ref{subsec:boundary-initial-conditions}.

\subsection{Classical solvability and positivity}
\label{subsec:classical-solvability-positivity}

We first record the local classical solvability and positivity property of
\eqref{eq:nondim-model}. The system is quasilinear only through the
chemotactic flux in the $v$-equation. Since the principal diffusion matrix is
triangular with positive diagonal entries $d_1,d_2,d_3$, the system is normally
parabolic in the standard sense. Local solvability therefore follows from
classical quasilinear parabolic theory \cite{amann1993nonhomogeneous}, while nonnegativity follows from the
quasi-positivity of the reaction terms and the no-flux boundary conditions.

\begin{theorem}[Local classical solvability and positivity]
\label{thm:local-solvability-positivity}
Let $\Omega\subset\mathbb R^d$ be a bounded domain with smooth boundary and let
$\theta\in(0,1)$. Suppose that
\[
(u_0,v_0,w_0)\in C^{2+\theta}(\overline\Omega)^3,
\qquad
u_0,v_0,w_0\ge0\quad\text{in }\overline\Omega,
\]
and assume that the initial data are compatible with the boundary conditions.
Then there exists a maximal time $T_{\max}\in(0,\infty]$
and a unique nonnegative classical solution
\[
(u,v,w)\in C^{1+\theta/2,\,2+\theta}_{\mathrm{loc}}
\big((0,T_{\max})\times\overline\Omega\big)^3
\]
of \eqref{eq:nondim-model}--\eqref{eq:nondim-noflux} satisfying
\[
u(0,\cdot)=u_0,\qquad
v(0,\cdot)=v_0,\qquad
w(0,\cdot)=w_0.
\]
Moreover, either $T_{\max}=\infty$, or
\[
\limsup_{t\uparrow T_{\max}}
\left(
\|u(t)\|_{L^\infty(\Omega)}
+\|v(t)\|_{L^\infty(\Omega)}
+\|w(t)\|_{L^\infty(\Omega)}
\right)=\infty .
\]
\end{theorem}

The proof, based on normal parabolicity and the quasi-positivity of the reaction
vector field, is given in
Appendices~\ref{app:normal-parabolicity}.

\subsection{A priori estimates}
\label{subsec:apriori-estimates}

We next state the estimates that are used throughout the rest of the paper.
The first estimate is a uniform pointwise bound on the tumour density; its proof
is given in Appendix~\ref{app:tumour-bound}.

\begin{proposition}[Tumour-density bound]
\label{prop:tumour-density-bound}
Let $(u,v,w)$ be a nonnegative classical solution of
\eqref{eq:nondim-model} on $[0,T_{\max})$. Then
\[
0\le u(t,x)\le M_u
\quad\text{for all }(t,x)\in[0,T_{\max})\times\overline\Omega,
\]
where $M_u:=\max\{1,\|u_0\|_{L^\infty(\Omega)}\}$.
\end{proposition}

The next estimates give global-in-time integral control of the immune and
chemokine variables as long as the classical solution exists; their proofs are
given in Appendices~\ref{app:v-mass}--\ref{app:w-mass}.

\begin{proposition}[Immune and chemokine mass estimates]
\label{prop:mass-estimates}
Let $(u,v,w)$ be a nonnegative classical solution of
\eqref{eq:nondim-model} on $[0,T_{\max})$, and let $M_u:=\max\{1,\|u_0\|_{L^\infty(\Omega)}\}$.
Then, for all $t<T_{\max}$,
\begin{equation}\label{eq:v-mass-differential}
\frac{d}{dt}\int_\Omega v(t,x)\,dx
\le
|\Omega|(\sigma_0+\sigma_1)-\delta\int_\Omega v(t,x)\,dx,
\end{equation}
and
\begin{equation}\label{eq:w-mass-differential}
\frac{d}{dt}\int_\Omega w(t,x)\,dx
\le
\alpha M_u|\Omega|
+\gamma M_u\int_\Omega v(t,x)\,dx
-\ell\int_\Omega w(t,x)\,dx .
\end{equation}
Consequently,
\begin{equation}\label{eq:v-mass-bound}
\int_\Omega v(t,x)\,dx
\le
e^{-\delta t}\int_\Omega v_0(x)\,dx
+
\frac{|\Omega|(\sigma_0+\sigma_1)}{\delta}
\bigl(1-e^{-\delta t}\bigr),
\end{equation}
and there exists a constant $C_w>0$, depending only on the parameters, $\Omega$,
and the initial masses, such that
\begin{equation}\label{eq:w-mass-bound}
\sup_{0<t<T_{\max}}\int_\Omega w(t,x)\,dx\le C_w .
\end{equation}
\end{proposition}

\begin{remark}[What the estimates do and do not prove]
\label{rem:no-global-linfty-overclaim}
The estimates above provide tumour-density boundedness and global mass control
for the immune and chemokine variables. They do not by themselves imply full
global $L^\infty$-boundedness of the three-component chemotaxis system in
arbitrary spatial dimension. Full boundedness would require additional
assumptions, such as stronger damping, small chemotactic sensitivity,
entropy-type estimates, signal-dependent desensitization, or
dimension-dependent regularity arguments \cite{winkler2010aggregation,liu2026computational,tao2012boundedness}; see
Appendix~\ref{app:no-global-linfty} for a precise discussion. Accordingly, the
results in the rest of the paper are stated for classical solutions on their
interval of existence, or for globally defined bounded solutions when such
boundedness is assumed or verified numerically.
\end{remark}

\section{Results (Equilibria)}
\label{sec:equilibria-biological-regimes}

In this section, we identify the spatially homogeneous equilibria of the
nondimensional tumour--immune--chemokine system \eqref{eq:nondim-model}.
Spatially homogeneous equilibria play two roles. Mathematically, they provide
the base states for the linear stability and dispersion-relation analysis.
Biologically, they correspond to different tumour--immune regimes: tumour-free
control, spatially homogeneous coexistence, or possible transition toward
spatially heterogeneous tumour--immune organization.

Throughout this section, a spatially homogeneous equilibrium is a constant
triple
\[
(u^*,v^*,w^*)\in[0,\infty)^3
\]
satisfying
\begin{equation}\label{eq:homogeneous-equilibrium-system}
\begin{cases}
0=u^*(1-u^*-v^*),\\[2pt]
0=\sigma_0+\sigma_1\frac{w^*}{1+w^*}-\delta v^*-\beta u^*v^*,\\[2pt]
0=\alpha u^*+\gamma u^*v^*-\ell w^*.
\end{cases}
\end{equation}

\subsection{Tumour-free equilibrium}
\label{subsec:tumour-free-equilibrium}

The tumour-free equilibrium is obtained by setting $u^*=0$. Since the chemokine
equation then gives $w^*=0$, the immune equation gives $v^*=\dfrac{\sigma_0}{\delta}$.
Thus,
\begin{equation}\label{eq:tumour-free-equilibrium}
E_0=\left(0,\frac{\sigma_0}{\delta},0\right).
\end{equation}
This equilibrium represents a tumour-free tissue state maintained by the
baseline immune supply $\sigma_0$ and immune loss rate $\delta$.

The stability of $E_0$ has a particularly simple interpretation. Near
$E_0$, the tumour equation linearizes as
\[
u_t=d_1\Delta u+\left(1-\frac{\sigma_0}{\delta}\right)u .
\]
Therefore, tumour perturbations decay when the baseline immune level
$\sigma_0/\delta$ exceeds the nondimensional tumour growth strength $1$, and
they grow when it is below this level.

\begin{theorem}[Tumour-free threshold]
\label{thm:tumour-free-threshold}
The tumour-free equilibrium $E_0$
is linearly stable if $\sigma_0>\delta$,
linearly unstable if $\sigma_0<\delta$,
and non-hyperbolic at $\sigma_0=\delta$.
At this equilibrium, chemotaxis does not affect the linear tumour-invasion
threshold.
\end{theorem}

\begin{proof}
Let
\[
u=0+\widetilde u,\qquad
v=\frac{\sigma_0}{\delta}+\widetilde v,\qquad
w=0+\widetilde w.
\]
The linearization of \eqref{eq:nondim-model} at $E_0$ is
\begin{equation}\label{eq:tumour-free-linearization}
\begin{cases}
\widetilde u_t=
d_1\Delta \widetilde u+
\left(1-\frac{\sigma_0}{\delta}\right)\widetilde u,\\[2pt]
\widetilde v_t=
d_2\Delta \widetilde v-\delta\widetilde v
+\sigma_1\widetilde w
-\xi\frac{\sigma_0}{\delta}\Delta\widetilde w
-\beta\frac{\sigma_0}{\delta}\widetilde u,\\[2pt]
\widetilde w_t=
d_3\Delta \widetilde w+
\left(\alpha+\gamma\frac{\sigma_0}{\delta}\right)\widetilde u
-\ell\widetilde w .
\end{cases}
\end{equation}
Using the Neumann Laplacian eigenbasis
\[
-\Delta\phi_k=\mu_k\phi_k,\qquad
\partial_n\phi_k=0,
\qquad
0=\mu_0<\mu_1\le\mu_2\le\cdots,
\]
the tumour component has growth rate
$\lambda_k^{(u)}
=
1-\dfrac{\sigma_0}{\delta}-d_1\mu_k$.
The largest of these rates occurs at $k=0$ and equals
$\lambda_0^{(u)}=1-\dfrac{\sigma_0}{\delta}$.
Hence all tumour perturbations decay if $\sigma_0>\delta$, while the homogeneous
tumour mode grows if $\sigma_0<\delta$.

The remaining two components form a triangularly forced stable subsystem once
$\widetilde u$ is fixed. Indeed, the diagonal decay rates for
$\widetilde v$ and $\widetilde w$ contain $-\delta-d_2\mu_k$ and
$-\ell-d_3\mu_k$, which are strictly negative. Therefore, the sign of the
dominant eigenvalue is determined by
$1-\sigma_0/\delta$. This proves the stated stability threshold.

Finally, the chemotactic contribution appears only in the immune equation
through the linear term
$-\xi\dfrac{\sigma_0}{\delta}\Delta\widetilde w$.
It does not enter the tumour growth rate
$\lambda_k^{(u)}$. Thus chemotaxis does not affect the linear
tumour-invasion threshold at $E_0$.
\end{proof}

\begin{remark}[Biological interpretation]
The threshold $\sigma_0>\delta$
means that baseline immune supply is sufficiently large relative to immune loss
to maintain a tumour-free state. Conversely, when $\sigma_0<\delta$,
the tumour-free state is linearly unstable, and a small tumour perturbation can
invade. Thus the ratio $\sigma_0/\delta$ acts as a nondimensional immune-control
threshold for tumour-free stability.

The tumour-free threshold $\sigma_0/\delta$ has a direct dimensional reading in tissue terms.
Restoring units, $\sigma_0/\delta=\kappa s_0/(r\mu)=\kappa E_0/r$, where
$E_0=s_0/\mu$ is the tumour-free immune density; thus $\sigma_0/\delta$ is the
ratio of the baseline immune-mediated tumour kill rate $\kappa E_0$ to the tumour
proliferation rate $r$. Control ($\sigma_0/\delta>1$) requires that resident
immune surveillance remove tumour cells at least as fast as they divide. Using
$r\sim0.1$--$0.5\ \mathrm{day}^{-1}$ and effective immune-mediated kill rates
$\kappa E_0\sim0.05$--$0.6\ \mathrm{day}^{-1}$ from validated tumour--immune
data \cite{de2005validated,yu2026chemotactic,kuznetsov1994nonlinear}, one finds
$\sigma_0/\delta$ of order $0.1$--$10$, i.e.\ straddling the control threshold.
This near-threshold position is biologically meaningful: it is precisely why
tumour fate is context-dependent, and why modest shifts in immune supply or
turnover---as in the transition from immune ``hot'' to ``cold'' tumours
\cite{galon2019approaches}---can move a lesion across the surveillance boundary.
\end{remark}

\subsection{Coexistence equilibria}
\label{subsec:coexistence-equilibria}

We next consider spatially homogeneous coexistence equilibria, for which
\[
u^*>0,\qquad v^*>0,\qquad w^*>0.
\]
From the first equilibrium equation in
\eqref{eq:homogeneous-equilibrium-system}, positivity of $u^*$ implies
$1-u^*-v^*=0$,
and therefore
\begin{equation}\label{eq:vstar-coexistence}
v^*=1-u^*.
\end{equation}
Hence a positive coexistence equilibrium must satisfy $0<u^*<1$.
The third equilibrium equation gives
\begin{equation}\label{eq:wstar-coexistence}
w^*
=
\frac{u^*[\alpha+\gamma(1-u^*)]}{\ell}.
\end{equation}
Substituting \eqref{eq:vstar-coexistence} and
\eqref{eq:wstar-coexistence} into the immune equilibrium equation yields a
single scalar equation for $u^*$:
\begin{equation}\label{eq:F-definition}
F(u)=0,\qquad 0<u<1,
\end{equation}
where
\begin{equation}\label{eq:F-function}
F(u)
=
\sigma_0+\sigma_1\frac{w(u)}{1+w(u)}
-(\delta+\beta u)(1-u),
\end{equation}
and
\begin{equation}\label{eq:w-of-u}
w(u)
=
\frac{u[\alpha+\gamma(1-u)]}{\ell}.
\end{equation}

\begin{proposition}[Coexistence equilibria]
\label{prop:coexistence-equilibria}
Every spatially homogeneous coexistence equilibrium satisfies
\[
v^*=1-u^*,\qquad
w^*=\frac{u^*[\alpha+\gamma(1-u^*)]}{\ell},
\]
where $u^*\in(0,1)$ is a zero of
\[
F(u)
=
\sigma_0+\sigma_1\frac{w(u)}{1+w(u)}
-(\delta+\beta u)(1-u),
\qquad
w(u)=\frac{u[\alpha+\gamma(1-u)]}{\ell}.
\]
If $\sigma_0<\delta$ and $F(1)>0$,
then at least one interior coexistence equilibrium $u^*\in(0,1)$ exists. The
condition $F(1)>0$ holds whenever $\sigma_0>0$, or whenever $\sigma_1>0$ and
$\alpha>0$.
\end{proposition}

\begin{proof}
Let $(u^*,v^*,w^*)$ be a coexistence equilibrium. Since $u^*>0$, the first
equation in \eqref{eq:homogeneous-equilibrium-system} gives
$v^*=1-u^*$. Since $v^*>0$, it follows that $u^*\in(0,1)$. The third equation
then gives
\[
w^*=\frac{\alpha u^*+\gamma u^*v^*}{\ell}
=
\frac{u^*[\alpha+\gamma(1-u^*)]}{\ell}.
\]
Substituting these expressions into the second equilibrium equation gives
$F(u^*)=0$, with $F$ defined by \eqref{eq:F-function}.

For existence, observe that $F$ is continuous on $[0,1]$, with
\[
w(0)=0,\qquad
F(0)=\sigma_0-\delta<0,
\]
under the hypothesis $\sigma_0<\delta$, while
\[
F(1)
=
\sigma_0+\sigma_1\frac{w(1)}{1+w(1)},
\qquad
w(1)=\frac{\alpha}{\ell}.
\]
If $F(1)>0$, the intermediate value theorem produces a zero $u^*\in(0,1)$.
Finally, $F(1)>0$ holds whenever $\sigma_0>0$; and if $\sigma_0=0$ then
$F(1)=\sigma_1\,(\alpha/\ell)/(1+\alpha/\ell)$, which is positive precisely when
$\sigma_1>0$ and $\alpha>0$.
\end{proof}

\begin{remark}[Boundary case]
If $\sigma_0=\sigma_1=0$, or if $\sigma_1>0$ but $\alpha=0$ with $\sigma_0=0$,
then $F(1)=0$ and the endpoint $u=1$ corresponds to an immune-free
tumour-dominant boundary state ($v^*=0$) rather than a positive coexistence
state.
\end{remark}

\begin{remark}[Biological interpretation]
A coexistence equilibrium represents a spatially homogeneous balance between
tumour growth, immune killing, immune recruitment, and chemokine signalling. The
condition $u^*\in(0,1)$ means that the tumour persists below carrying capacity,
while $v^*=1-u^*$ means that immune pressure offsets the remaining tumour growth
capacity. The scalar equation $F(u)=0$ balances immune recruitment,
immune loss, and tumour-induced immune suppression. Thus coexistence is governed
by the competition between immune supply and chemokine recruitment on one side,
and immune decay and tumour-associated immune loss on the other.
\end{remark}

\section{Results (Linear Stability Analysis)}
\label{sec:linear-stability}

We now analyse the spatial stability of a homogeneous coexistence equilibrium
$(u^*,v^*,w^*)$ of Proposition~\ref{prop:coexistence-equilibria} with respect to
finite-wavelength perturbations. This is the mechanism through which chemotaxis
can convert a spatially homogeneous tumour--immune state into a spatially
heterogeneous one. The algebraic details are collected in
Appendix~\ref{app:routh-hurwitz}; here we record the mode-wise stability matrix,
the dispersion relation, and the resulting notion of a critical chemotactic
sensitivity.

\subsection{Mode-wise stability matrix}
\label{subsec:Mk}

\begin{lemma}[The linearized flux condition reduces to Neumann]
\label{lem:bc-reduction}
Let $(u^*,v^*,w^*)$ be a homogeneous equilibrium and let
$(\widetilde u,\widetilde v,\widetilde w)$ be a perturbation satisfying the
linearization of the physical no-flux conditions \eqref{eq:noflux}. Then the
linearized conditions are exactly the componentwise Neumann conditions
\[
\nabla\widetilde u\cdot n=\nabla\widetilde v\cdot n=\nabla\widetilde w\cdot n=0
\qquad\text{on }\partial\Omega .
\]
\end{lemma}

\begin{proof}
The tumour and chemokine conditions in \eqref{eq:noflux} are already
$\nabla u\cdot n=0$ and $\nabla w\cdot n=0$, which linearize to
$\nabla\widetilde u\cdot n=0$ and $\nabla\widetilde w\cdot n=0$. The immune
condition is $(d_2\nabla v-\xi v\nabla w)\cdot n=0$. Linearizing about
$(v^*,w^*)$ with $\nabla w^*=0$ gives
\[
\big(d_2\nabla\widetilde v-\xi v^*\nabla\widetilde w-\xi\widetilde v\,\nabla w^*\big)\cdot n
= d_2\,\nabla\widetilde v\cdot n-\xi v^*\,\nabla\widetilde w\cdot n .
\]
The second term vanishes because $\nabla\widetilde w\cdot n=0$ from the chemokine
condition, leaving $d_2\,\nabla\widetilde v\cdot n=0$, i.e.\
$\nabla\widetilde v\cdot n=0$ since $d_2>0$. Thus the coupled flux condition is
equivalent, at the linear level about a homogeneous state, to the separate
Neumann conditions used in the spectral analysis. In particular the Neumann
Laplacian eigenbasis $\{\phi_k\}$ diagonalizes all three linearized fluxes
simultaneously, which is what makes the mode-by-mode reduction to $M_k$ in \eqref{eq:Mk-matrix}
legitimate.
\end{proof}

Writing $u=u^*+\widetilde u$, $v=v^*+\widetilde v$, $w=w^*+\widetilde w$ and
keeping first-order terms, the chemotactic flux linearizes as
\[
-\xi\nabla\cdot(v\nabla w)
\;\longmapsto\;
-\xi\nabla\cdot(v^*\nabla\widetilde w)
=
-\xi v^*\Delta\widetilde w,
\]
because $\nabla w^*=0$ at a homogeneous state. Projecting the linearized system
onto the Neumann eigenbasis $-\Delta\phi_k=\mu_k\phi_k$, $\partial_n\phi_k=0$
(so that $\mu_k=(k\pi/L)^2$ in one dimension and $\mu_0=0$), and using
$1-2u^*-v^*=-u^*$ at coexistence, the modal amplitudes
$(U_k,V_k,W_k)$ satisfy $\tfrac{d}{dt}(U_k,V_k,W_k)^\top
=M_k\,(U_k,V_k,W_k)^\top$ with
\begin{equation}\label{eq:Mk-matrix}
M_k=
\begin{pmatrix}
-u^*-d_1\mu_k & -u^* & 0\\[2mm]
-\beta v^* & -(\delta+\beta u^*)-d_2\mu_k &
\dfrac{\sigma_1}{(1+w^*)^2}+\xi v^*\mu_k\\[3mm]
\alpha+\gamma v^* & \gamma u^* & -\ell-d_3\mu_k
\end{pmatrix}.
\end{equation}
The chemotactic sensitivity $\xi$ enters $M_k$ only through the $(2,3)$ entry, as
the term $+\xi v^*\mu_k$ obtained from
$-\xi v^*\Delta\widetilde w=\xi v^*\mu_k\widetilde w$. This contribution vanishes
for the homogeneous mode $\mu_0=0$ and grows linearly with $\mu_k$; it is the
sole route by which chemotaxis can destabilize finite spatial modes while leaving
the spatially uniform mode unaffected.

\subsection{Dispersion relation and critical sensitivity}
\label{subsec:dispersion-relation}

The equilibrium is linearly stable with respect to mode $k$ if and only if all
eigenvalues of $M_k$ lie in the open left half-plane. We define the dispersion
relation
\begin{equation}\label{eq:dispersion}
\omega_k(\xi)=\max_j\operatorname{Re}\lambda_j(M_k),
\end{equation}
so that the coexistence state is linearly stable to all spatial perturbations
precisely when $\omega_k(\xi)<0$ for every $k\ge0$, and the critical chemotactic
sensitivity is
\begin{equation}\label{eq:xic-definition}
\xi_c=\inf\Bigl\{\xi\ge0:\ \max_{k\ge1}\omega_k(\xi)=0\Bigr\}.
\end{equation}
The characteristic polynomial of $M_k$ is the monic cubic $\lambda^3+a_1(k)\lambda^2+a_2(k)\lambda+a_3(k)$, whose coefficients are computed
in Appendix~\ref{app:routh-hurwitz}. We have the following result:
\begin{proposition}[Existence of a critical sensitivity and nature of the onset]
\label{prop:xic-existence}
Assume the coexistence equilibrium $(u^*,v^*,w^*)$ is stable to the homogeneous
mode, i.e.\ all eigenvalues of $M_0$ have negative real part, and adopt the
notation \eqref{eq:app-rh-notation}--\eqref{eq:app-rh-notation-2}. For each mode
$k\ge1$ the coefficients $a_1(k),a_2(k),a_3(k)$ of the characteristic cubic are
smooth, and $a_1(k)>0$. Then:
\begin{enumerate}
\item[(i)] \emph{(Monotone destabilization.)} Both $a_2(k;\xi)$ and $a_3(k;\xi)$
are affine and strictly decreasing in $\xi$, with slopes
$\partial_\xi a_2(k)=-v^*\mu_k\,t<0$ and
$\partial_\xi a_3(k)=-v^*\mu_k\,(A_k t-ps)$. Consequently the Hurwitz
determinant $H_k(\xi):=a_1(k)a_2(k;\xi)-a_3(k;\xi)$ is affine in $\xi$ and, for
each fixed $k$ with $\mu_k>0$, the map $\xi\mapsto\omega_k(\xi)$ is continuous and
eventually positive.
\item[(ii)] \emph{(Existence of $\xi_c$.)} Since $\omega_k(0)<0$ for all $k$
(homogeneous stability plus continuity of eigenvalues) and
$\omega_{k}(\xi)\to+\infty$ as $\xi\to\infty$ for any $k$ with $\mu_k>0$, the set
$\{\xi\ge0:\max_{k\ge1}\omega_k(\xi)=0\}$ is nonempty and
$\xi_c=\inf\{\xi:\max_{k\ge1}\omega_k(\xi)\ge0\}\in(0,\infty)$ is attained at some
finite mode $k^*\ge1$.
\item[(iii)] \emph{(Stationary vs.\ oscillatory.)} At $\xi=\xi_c$ the critical
mode $k^*$ satisfies exactly one of:
\begin{enumerate}
\item[(a)] $a_3(k^*;\xi_c)=0$, $a_1a_2-a_3>0$ (stationary/Turing: a simple real eigenvalue crosses $0$),
\item[(b)] $H_{k^*}(\xi_c)=a_1a_2-a_3=0$, $a_3>0$ (oscillatory/Hopf: a complex pair crosses the imaginary axis at $\pm i\sqrt{a_2})$.
\end{enumerate}
The onset is stationary iff $a_3(k^*;\xi_c)=0$ occurs at no larger $\xi$ than
$H_{k^*}=0$, and oscillatory otherwise. In case (b) the marginal frequency is
$\omega^\ast=\sqrt{a_2(k^*;\xi_c)}$.
\end{enumerate}
\end{proposition}

\begin{proof}
(i) From \eqref{eq:app-a2}--\eqref{eq:app-a3} and
$r_k=\sigma_1/(1+w^*)^2+\xi v^*\mu_k$, only $r_k$ depends on $\xi$, linearly, so
$a_2$ and $a_3$ are affine in $\xi$ with the stated slopes; $\partial_\xi a_2=-t\,v^*\mu_k<0$
because $t=\gamma u^*>0$, and $\partial_\xi a_3=-v^*\mu_k(A_kt-ps)$ from
\eqref{eq:app-a3-xi-negative-form}. Hence $H_k$ is affine in $\xi$.
(ii) Homogeneous stability gives $\omega_k(0)<0$ for every $k$ (the reaction
Jacobian equals $M_0$ shifted by $-\mathrm{diag}(d_i)\mu_k\prec0$, so adding
diffusion cannot destabilize at $\xi=0$). As $\xi\to\infty$, the $(2,3)$ entry
$+\xi v^*\mu_k\to\infty$ forces an eigenvalue of $M_k$ into the right half-plane
(e.g.\ $a_3\to-\infty$ when $A_kt-ps>0$, or $H_k\to-\infty$ when
$\partial_\xi H_k<0$), so $\omega_k(\xi)\to+\infty$. Continuity of eigenvalues
and compactness of the neutral set over the finitely many modes with
$\mu_k$ below the eventual cutoff give a finite minimizing $k^*$ and
$\xi_c\in(0,\infty)$.
(iii) By the Routh--Hurwitz criterion for a cubic with $a_1>0$, the boundary of
the stable region is reached either when $a_3=0$ (a real root crosses zero,
$\lambda=0$) or when $a_1a_2-a_3=0$ with $a_3>0$ (a conjugate pair
$\lambda=\pm i\sqrt{a_2}$ crosses the imaginary axis); these are mutually
exclusive at a first crossing, and the affine (hence monotone) dependence in (i)
selects whichever occurs at the smaller $\xi$. The frequency follows from
substituting $\lambda=i\omega^\ast$ into the cubic, which gives
$\omega^{\ast2}=a_2$.
\end{proof}

\begin{remark}[Application to the computational parameter set]
\label{rem:hopf-onset-quant}
For Table~\ref{tab:experiment-parameters} one has $A_kt-ps<0$ throughout the
unstable band, so $a_3(k;\xi)$ stays positive and case~(b) applies: the onset is
oscillatory. The first admissible Neumann mode to lose stability is $k^*=5$, at
$\xi_c\approx12.48$ (the continuous-$\mu$ infimum reported in
Experiment~\ref{subsec:exp3-dispersion} is $12.47$, the two agreeing to within
the discrete mode spacing $\mu_{k+1}-\mu_k$). At this onset the critical mode has
a marginal complex pair $\lambda=\pm i\,\omega^\ast$ with $\omega^\ast\approx0.99$,
the Hopf condition $a_1a_2-a_3=0$ holds \emph{exactly} while $a_3\approx4.68>0$,
and consistently $\omega^\ast=\sqrt{a_2}$ and the third eigenvalue is real at
$-a_1\approx-4.78$. All five quantities are checked in the onset-diagnostics block
of Experiment~\ref{subsec:exp3-dispersion} (computed values
$\xi_c=12.479$, $k^*=5$, $\omega^\ast=0.989$, $a_1a_2-a_3=-1.5\times10^{-11}$,
$a_3=4.681$). This is the statement about the Hopf-type onset, and it is why
the onset must be located from $\omega_k$ directly rather than from the
constant-term threshold $\xi_k^*$. Note that this $k^*=5$ is the mode that first
goes unstable at $\xi_c$; the dominant (fastest-growing) mode at the supercritical
$\xi=1.6\,\xi_c$ used in Experiment~\ref{subsec:exp3-dispersion} is instead
$k=6$, and the two need not coincide.
\end{remark}

By the Routh--Hurwitz criterion, mode $k$ is
stable if and only if
\begin{equation}\label{eq:rh-mainbody}
a_1(k)>0,\qquad a_2(k)>0,\qquad a_3(k)>0,\qquad a_1(k)a_2(k)>a_3(k).
\end{equation}
Because the diffusion and reaction contributions make $a_1(k)>0$ automatically,
the loss of stability occurs through $a_3(k)=0$ (a stationary, Turing-type
crossing) or through $a_1(k)a_2(k)=a_3(k)$ (an oscillatory, Hopf-type crossing);
this mode-projection approach to diffusion-driven instability is classical in
reaction--diffusion pattern formation \cite{murray2003mathematical}.
The chemotactic sensitivity enters both $a_2(k)$ and $a_3(k)$ through the
amplified coupling $r_k=\sigma_1/(1+w^*)^2+\xi v^*\mu_k$. Appendix~\ref{app:rh-chemotaxis-constant-term}
gives an explicit constant-term destabilization threshold $\xi_k^*$ that is valid
only when a sign condition holds; since the first crossing may instead be of
Hopf type, the onset is most safely located from \eqref{eq:dispersion}--\eqref{eq:xic-definition}
directly, and that is the route used in the numerical experiments of
Section~\ref{sec:numerical-experiments}.

The destabilizing coupling combines two effects. The recruitment sensitivity
$\sigma_1/(1+w^*)^2$ measures how strongly local immune recruitment responds to
the chemokine concentration, while the spatial term $\xi v^*\mu_k$ measures the
chemotactic amplification of the immune response to a chemokine mode of
wavenumber $\mu_k$. Because the latter grows with $\mu_k$, a sufficiently large
chemotactic sensitivity destabilizes a band of finite wavelengths even when the
homogeneous coexistence state is stable to spatially uniform perturbations. This
is the mathematical mechanism behind the transition from homogeneous
tumour--immune coexistence to spatially heterogeneous immune infiltration or
exclusion patterns, which is verified numerically in
Section~\ref{sec:numerical-experiments}.

\section{Results (Numerical Experiments)}
\label{sec:numerical-experiments}

The numerical experiments in this section are organized around the theoretical
results of Sections~\ref{sec:mathematical-analysis}--\ref{sec:equilibria-biological-regimes}.
Rather than presenting simulations as illustrations, each experiment is designed
to verify a specific analytical statement and to quantify the agreement
between theory and computation. Concretely, we test the tumour-free threshold of
Theorem~\ref{thm:tumour-free-threshold} (Experiment~1), the coexistence
construction of Proposition~\ref{prop:coexistence-equilibria}
(Experiment~2), the chemotaxis dispersion relation derived below
(Experiment~3), the chemotaxis-driven onset of spatial heterogeneity
(Experiment~4), the parametric structure of the stability boundary
(Experiment~5), and the convergence, positivity, and residual consistency of the
finite-volume scheme of Section~\ref{sec:numerical-method} (Experiment~6).

All computations use the conservative finite-volume discretization of
Section~\ref{sec:numerical-method} with upwind chemotactic flux, discrete no-flux
boundaries, and implicit (BDF) time stepping, on the one-dimensional domain
$\Omega=(0,L)$ with $L=10$. Unless stated otherwise the parameters are those of
Table~\ref{tab:experiment-parameters}, for which the homogeneous coexistence
equilibrium of Proposition~\ref{prop:coexistence-equilibria} is
\begin{equation}\label{eq:numeric-coexistence}
(u^*,v^*,w^*)=(0.5311,\,0.4689,\,0.3900),
\end{equation}
obtained as the unique interior root of $F(u)=0$ in \eqref{eq:F-function}. This
equilibrium is stable to spatially homogeneous perturbations (the eigenvalues of
the reaction Jacobian are $-1.302$ and $-0.371\pm0.211\,i$), so any instability
reported below is genuinely chemotaxis- and wavenumber-driven.

\begin{table}[htbp]
\centering
\caption{Baseline nondimensional parameters used in the numerical experiments.
Experiment~1 instead varies $\sigma_0$ with $\delta=0.30$ fixed; Experiment~5
varies the indicated pairs.}
\label{tab:experiment-parameters}
\begin{tabular}{llllll}
\toprule
$d_1$ & $d_2$ & $d_3$ & $\sigma_0$ & $\sigma_1$ & $\delta$ \\
$0.01$ & $0.10$ & $1.00$ & $0.10$ & $0.50$ & $0.30$ \\
\midrule
$\beta$ & $\alpha$ & $\gamma$ & $\ell$ & $L$ & \\
$0.40$ & $0.50$ & $0.50$ & $1.00$ & $10$ & \\
\bottomrule
\end{tabular}
\end{table}

The linear-stability experiments use the mode-wise stability matrix $M_k$ of
\eqref{eq:Mk-matrix}, derived in Section~\ref{sec:linear-stability} by
linearizing about the homogeneous coexistence equilibrium and projecting onto the
Neumann Laplacian eigenbasis $-\Delta\phi_k=\mu_k\phi_k$ (so that
$\mu_k=(k\pi/L)^2$ in one dimension). We recall that the chemotactic sensitivity
enters $M_k$ only through the $(v,w)$ entry as $+\xi v^*\mu_k$, that the
dispersion relation is $\omega_k(\xi)=\max_j\operatorname{Re}\lambda_j(M_k)$, and
that the critical sensitivity $\xi_c$ is the smallest $\xi$ for which
$\max_{k\ge1}\omega_k(\xi)=0$.

\subsection{Experiment 1: tumour-free threshold}
\label{subsec:exp1-tumour-free}

We use this experiment to verify Theorem~\ref{thm:tumour-free-threshold}: the tumour-free equilibrium
$E_0=(0,\sigma_0/\delta,0)$ is linearly stable for $\sigma_0>\delta$ and unstable
for $\sigma_0<\delta$, with linear tumour growth rate $1-\sigma_0/\delta$
independent of chemotaxis.

We fix $\delta=0.30$ and take a small tumour seed
$u_0=10^{-3}(1+0.3\cos(2\pi x/L))$ on the tumour-free background
$v_0=\sigma_0/\delta$, $w_0=0$, with $\xi=0$. Two regimes are compared:
the stable case $\sigma_0=0.50>\delta$ and the invasion case
$\sigma_0=0.10<\delta$.

Figure~\ref{fig:exp1} reports $\|u(t)\|_\infty$, the tumour mass
$\int_\Omega u\,dx$, and the measured exponential rate against the predicted
value $1-\sigma_0/\delta$. In the stable regime the tumour decays to machine
zero at the rate $-0.6677$ (predicted $-0.6667$); in the invasion regime it grows
at $0.6635$ (predicted $0.6667$) and saturates at the coexistence value
$u^*=0.531$ from \eqref{eq:numeric-coexistence}. The measured rates lie on the
diagonal of Figure~\ref{fig:exp1}(c), confirming the threshold and the
predicted growth rate to within the time-stepping tolerance.

\begin{figure}[htbp]
\centering
\includegraphics[width=\textwidth]{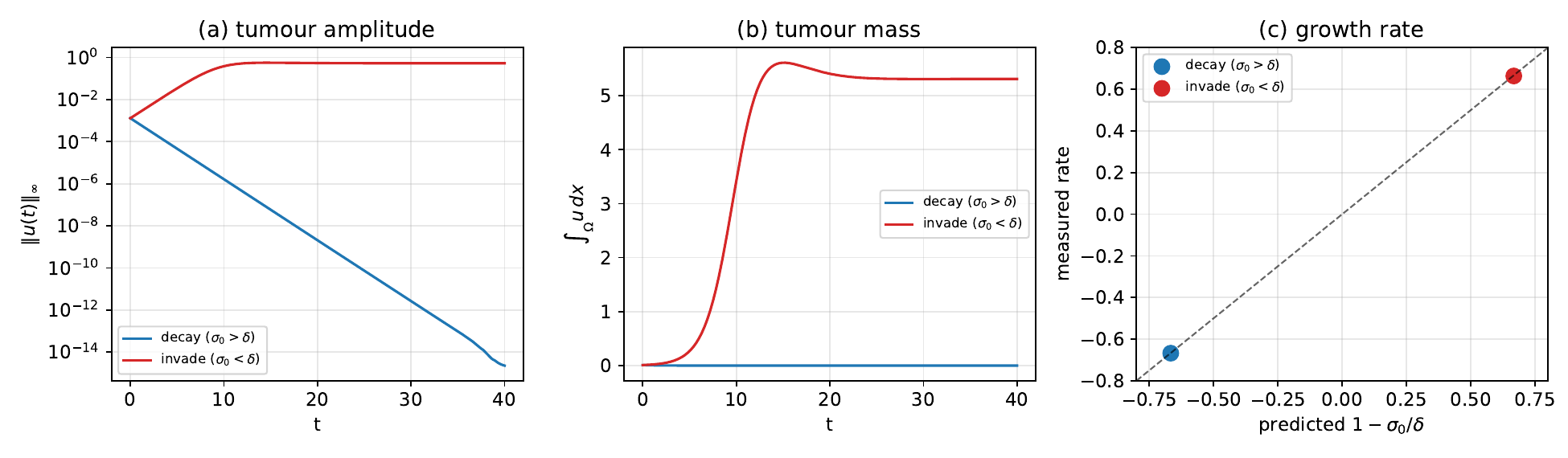}
\caption{Experiment~1. (a) Tumour amplitude $\|u(t)\|_\infty$ for the decay
regime $\sigma_0>\delta$ and the invasion regime $\sigma_0<\delta$;
(b) tumour mass $\int_\Omega u\,dx$; (c) measured versus predicted growth rate
$1-\sigma_0/\delta$ (dashed line is identity). The measured rates agree with
Theorem~\ref{thm:tumour-free-threshold}.}
\label{fig:exp1}
\end{figure}
\clearpage

\subsection{Experiment 2: coexistence equilibrium verification}
\label{subsec:exp2-coexistence}

We use this experiment to verify Proposition~\ref{prop:coexistence-equilibria}: starting from generic
positive data, the dynamics (with $\xi=0$, below any patterning threshold)
converge to the homogeneous coexistence state determined by the scalar equation
$F(u^*)=0$, with $v^*=1-u^*$ and $w^*=u^*[\alpha+\gamma(1-u^*)]/\ell$.

We integrate from spatially random nonnegative initial data
$u_0,v_0,w_0\in[0.2,0.5]$ to $t=200$ and compare the numerical steady state with
the analytical equilibrium \eqref{eq:numeric-coexistence}.

Figure~\ref{fig:exp2} shows monotone convergence of
$\|(u,v,w)-(u^*,v^*,w^*)\|_{L^2}$ to $4.4\times10^{-11}$, and the final spatial
fields coincide with the analytical equilibrium (dashed) to the same order. The
componentwise errors are $1.1\times10^{-11}$, $4.7\times10^{-12}$, and
$7.7\times10^{-12}$ for $u$, $v$, $w$ respectively, confirming both the existence
construction and the closed-form expressions for $v^*$ and $w^*$.

\begin{figure}[htbp]
\centering
\includegraphics[width=\textwidth]{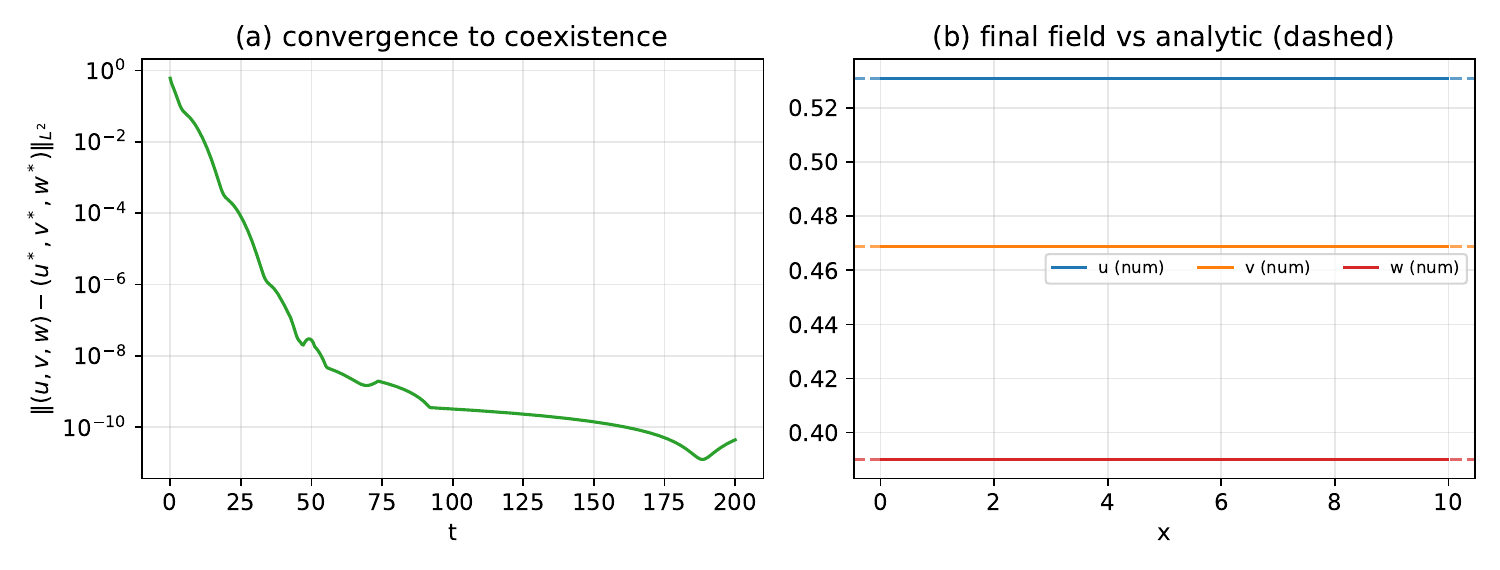}
\caption{Experiment~2. (a) Convergence of the solution to the analytical
coexistence equilibrium in $L^2$; (b) final numerical fields (solid) versus the
analytical values $(u^*,v^*,w^*)$ of \eqref{eq:numeric-coexistence} (dashed).}
\label{fig:exp2}
\end{figure}
\clearpage

\subsection{Experiment 3: dispersion relation and unstable modes}
\label{subsec:exp3-dispersion}

We use this experiment to verify the dispersion relation associated with \eqref{eq:Mk-matrix}: the
eigenvalues of $M_k$ predict which spatial modes grow, the dominant wavenumber,
and the per-mode growth rates measured from the full nonlinear simulation.

About the coexistence equilibrium \eqref{eq:numeric-coexistence} we compute
$\max_j\operatorname{Re}\lambda_j(M_k)$ as a function of the wavenumber index
$k$ for several values of $\xi$. The smallest sensitivity at which a finite-$k$
mode becomes neutral is $\xi_c\approx 12.47$.
For $\xi=1.6\,\xi_c$ we then integrate the nonlinear system from a small random
perturbation of the equilibrium, extract the discrete-cosine (Neumann-mode)
amplitudes of $v$, and fit their exponential growth rate during the linear phase.

Figure~\ref{fig:exp3}(a) shows that the homogeneous and weakly chemotactic cases
are stable for all $k$, while for $\xi=1.6\,\xi_c$ a band of intermediate
wavenumbers becomes unstable, with maximal growth near $k^*=6$. Panel~(b)
compares the linear-theory growth rates with the simulated modal growth rates:
in the active (near-marginal and unstable) band the agreement is excellent, with
correlation $0.980$, and the dominant mode matches in both location ($k^*=6$) and
rate (predicted $0.267$, measured $0.273$). Strongly damped modes saturate at the
nonlinear noise floor rather than following their large negative linear rates,
as expected, since linear theory governs only the growing band.

\begin{figure}[htbp]
\centering
\includegraphics[width=\textwidth]{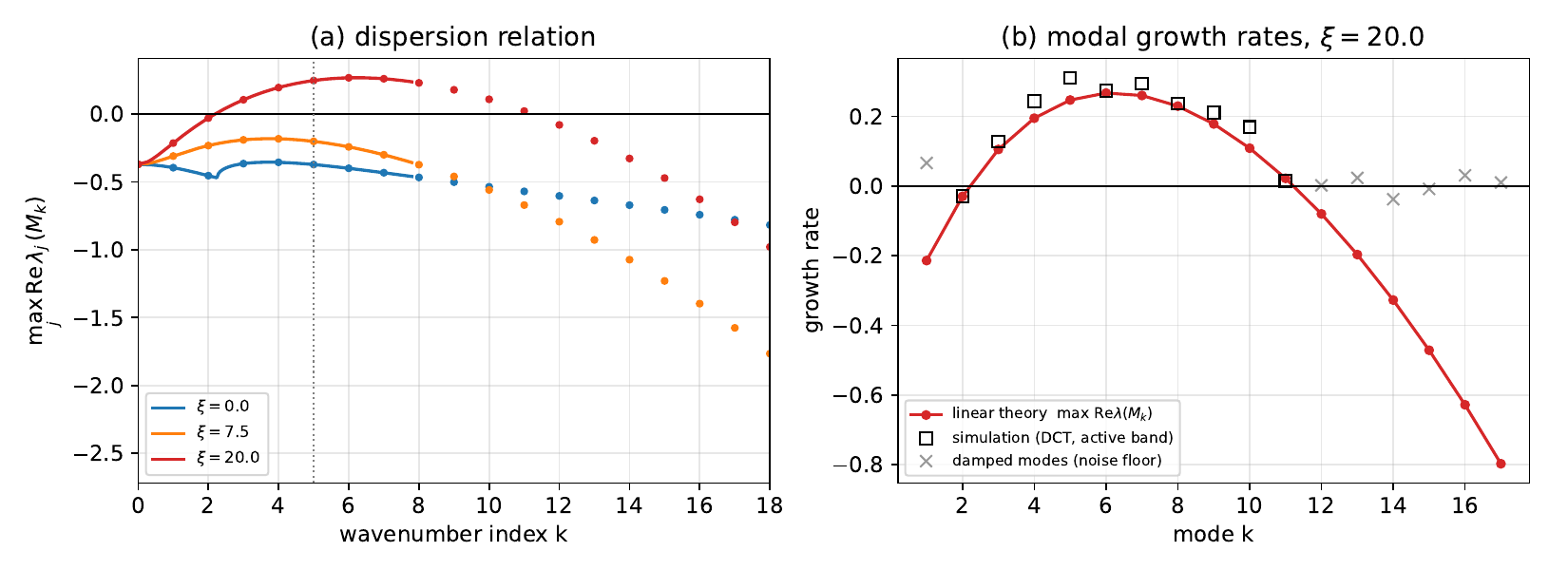}
\caption{Experiment~3. (a) Dispersion relation
$\max_j\operatorname{Re}\lambda_j(M_k)$ versus wavenumber index $k$ for
$\xi\in\{0,\,0.6\,\xi_c,\,1.6\,\xi_c\}$ (curves: continuous $\mu$; markers:
admissible Neumann modes). (b) Linear-theory modal growth rates (line) versus
rates measured from the nonlinear simulation via the discrete cosine transform
(squares: active band; crosses: damped modes at the noise floor) for
$\xi=1.6\,\xi_c$.}
\label{fig:exp3}
\end{figure}
\clearpage

\subsection{Experiment 4: chemotaxis-driven pattern formation}
\label{subsec:exp4-patterns}

We use this experiment to verify that $\xi_c$ separates spatially homogeneous behaviour from spatially
heterogeneous tumour--immune organization: convergence to the homogeneous
coexistence state for $\xi<\xi_c$, marginal small-amplitude structure near
$\xi_c$, and a sustained pattern for $\xi>\xi_c$.

Starting from the same perturbed coexistence data, we integrate to $t=120$ for
$\xi=0.7\,\xi_c$ (subcritical), $\xi\approx\xi_c$ (onset), and
$\xi=1.8\,\xi_c$ (supercritical), tracking the spatial amplitude
$\max_x u-\min_x u$.

Figure~\ref{fig:exp4} confirms the predicted transition. For $\xi=0.7\,\xi_c$ the
amplitude decays to $\sim10^{-9}$ and the fields return to homogeneity; near
$\xi_c$ a small steady modulation of amplitude $\approx 5\times10^{-3}$ persists;
for $\xi=1.8\,\xi_c$ the system develops a large-amplitude pattern in which the
immune density forms sharp aggregation peaks at chemokine maxima
($\max u-\min u\approx0.69$, $\max v-\min v\approx6.9$) with the tumour
correspondingly depleted there. The minimum of every component remains strictly
positive throughout (e.g.\ $\min u\approx0.10$, $\min v\approx0.06$ in the
supercritical run), consistent with the positivity property of
Theorem~\ref{thm:local-solvability-positivity} and the positivity-preserving
design of the scheme.

\begin{figure}[htbp]
\centering
\includegraphics[width=\textwidth]{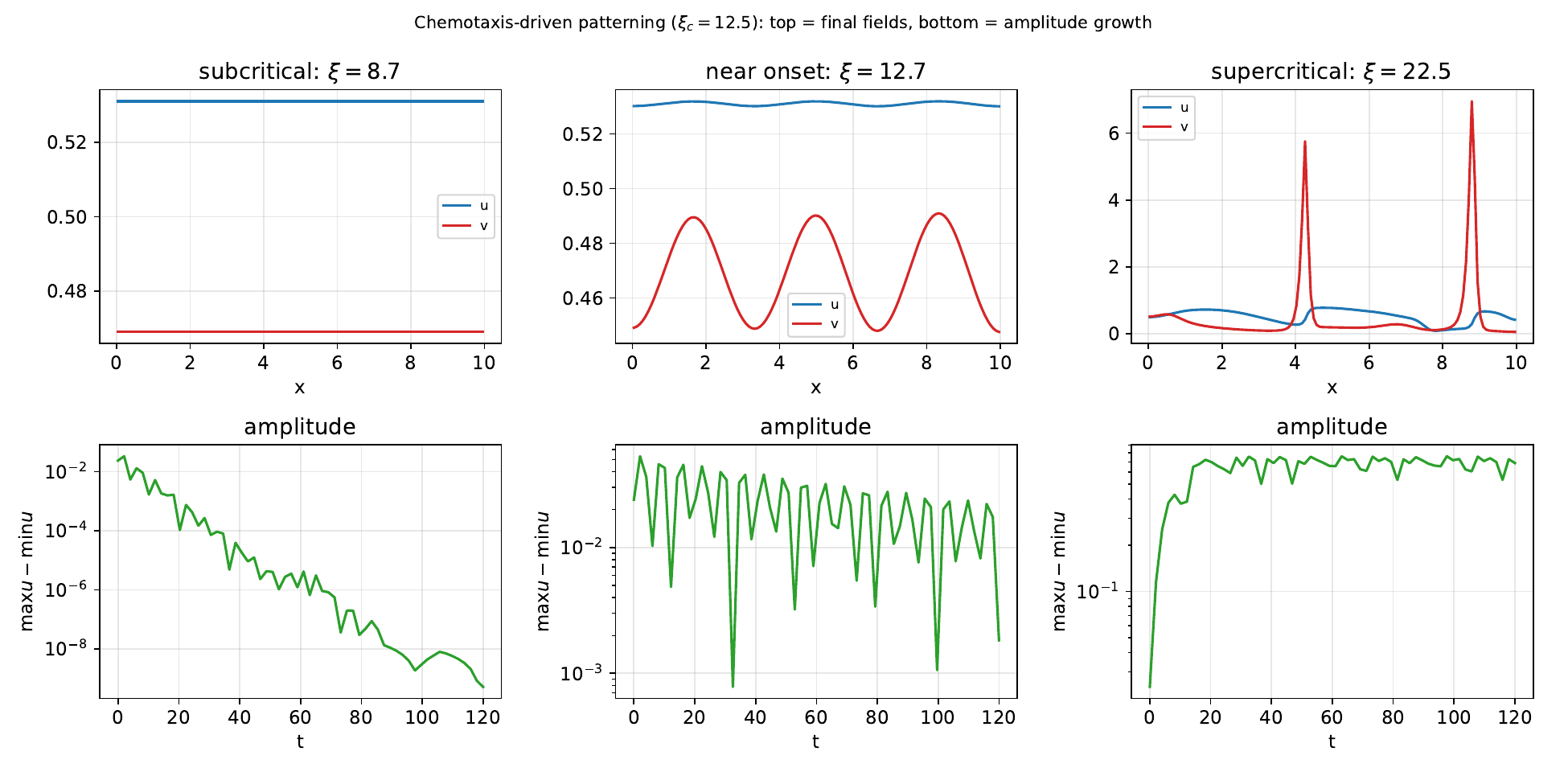}
\caption{Experiment~4. Top row: final tumour ($u$) and immune ($v$) fields for
subcritical, near-critical, and supercritical chemotactic sensitivity. Bottom
row: corresponding spatial amplitude $\max_x u-\min_x u$ versus time. The
transition homogeneous $\to$ onset $\to$ heterogeneity occurs across
$\xi_c\approx12.5$.}
\label{fig:exp4}
\end{figure}
\clearpage

\subsection{Experiment 5: two-parameter sensitivity maps}
\label{subsec:exp5-sensitivity}

We use this experiment to characterize how the patterning threshold depends on the model parameters by
mapping the stability boundary in the planes $(\xi,\sigma_1)$, $(\xi,\delta)$,
and $(\xi,d_2)$. This directly addresses sensitivity analysis of the
chemotaxis-driven instability.

For each point of a parameter grid we recompute the coexistence equilibrium from
$F(u^*)=0$ and evaluate
$\max_{k\ge1}\max_j\operatorname{Re}\lambda_j(M_k)$ over the admissible Neumann
modes. The zero level set of this quantity is the onset curve.

Figure~\ref{fig:exp5} shows the three stability maps. The instability region
(positive growth, red) expands as the chemokine-induced recruitment $\sigma_1$
increases, so that stronger recruitment lowers the critical sensitivity $\xi_c$.
Conversely, increasing the immune decay $\delta$ or the immune diffusion $d_2$
shrinks the unstable region and raises $\xi_c$: immune turnover and random
immune motility are both stabilizing, since they damp the spatial immune
aggregation that drives the instability. The monotone, smooth dependence of the
onset curve on each parameter quantifies the competition between chemotactic
aggregation and the stabilizing effects of immune loss and diffusion.

\begin{figure}[htbp]
\centering
\includegraphics[width=\textwidth]{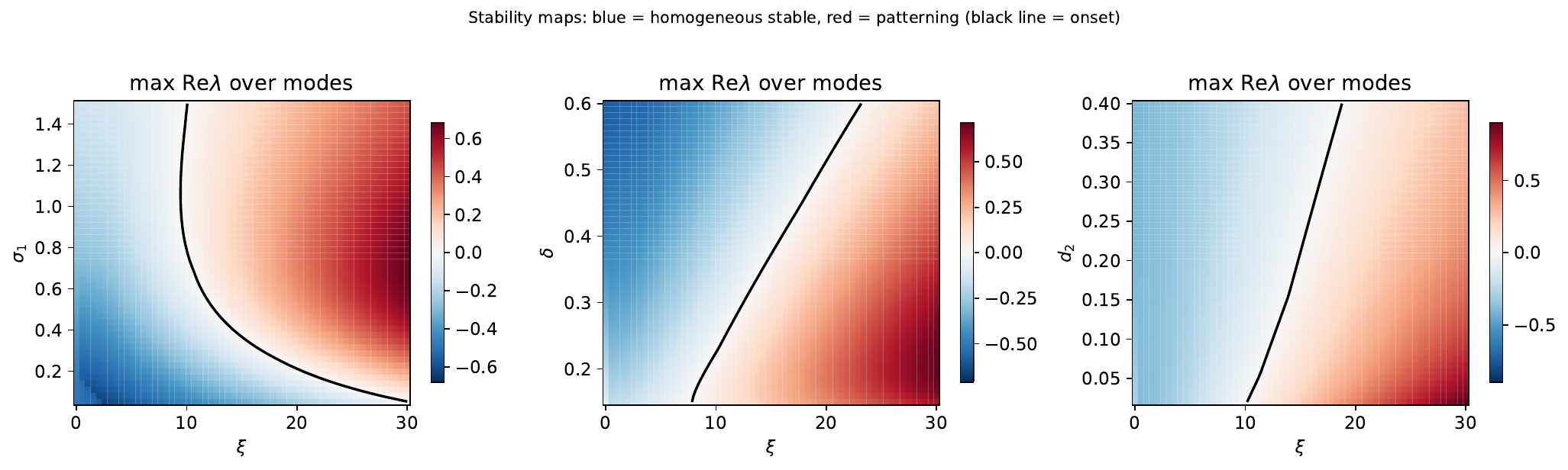}
\caption{Experiment~5. Stability maps of
$\max_{k\ge1}\max_j\operatorname{Re}\lambda_j(M_k)$ in the planes
$(\xi,\sigma_1)$, $(\xi,\delta)$, and $(\xi,d_2)$. Blue: homogeneous coexistence
stable; red: chemotaxis-driven patterning; black curve: onset
($\operatorname{Re}\lambda=0$). The equilibrium is recomputed at every grid
point.}
\label{fig:exp5}
\end{figure}
\clearpage

\subsection{Experiment 6: grid convergence and residual certification}
\label{subsec:exp6-convergence}

We use this experiment to certify that the computed patterns are properties of the PDE system and not
artifacts of the discretization, by demonstrating grid convergence, discrete
positivity, and small post-hoc PDE residuals
\eqref{eq:residual-u}--\eqref{eq:residual-w}.

In a mildly supercritical regime ($\xi=1.3\,\xi_c$, smooth low-amplitude
pattern) we integrate to $t=60$ on grids $N\in\{64,128,256,512\}$. We report the
successive differences $\|U_h-U_{h/2}\|_{L^2}$ (fine fields averaged to the
coarse grid), the running minimum of each component, and the residual norms
$\|R_u\|,\|R_v\|,\|R_w\|$ evaluated on the trajectory.

Figure~\ref{fig:exp6}(a) shows that $\|U_h-U_{h/2}\|_{L^2}$ decreases from
$5.0\times10^{-5}$ to $3.5\times10^{-6}$ as $N$ increases, an observed
convergence order of $1.92$. The order is close to two because the diffusion and
reaction terms are second-order accurate and dominate here; the first-order
upwind chemotactic flux reduces the order toward one only in strongly
chemotaxis-dominated regimes with steep immune gradients. The running minimum of
each component remains positive on every grid (uniformly $\approx0.389$,
consistent with the small pattern amplitude), confirming discrete positivity.
Figure~\ref{fig:exp6}(b) shows that the PDE residuals remain small and decay in
time (maxima $\|R_u\|\approx2\times10^{-3}$, $\|R_v\|\approx1\times10^{-2}$,
$\|R_w\|\approx6\times10^{-4}$), certifying that the computed solution is a
faithful discrete realization of \eqref{eq:nondim-model}. Together the
convergence, positivity, and residual diagnostics validate the spatial patterns
reported in Experiments~3 and~4.

\begin{figure}[htbp]
\centering
\includegraphics[width=\textwidth]{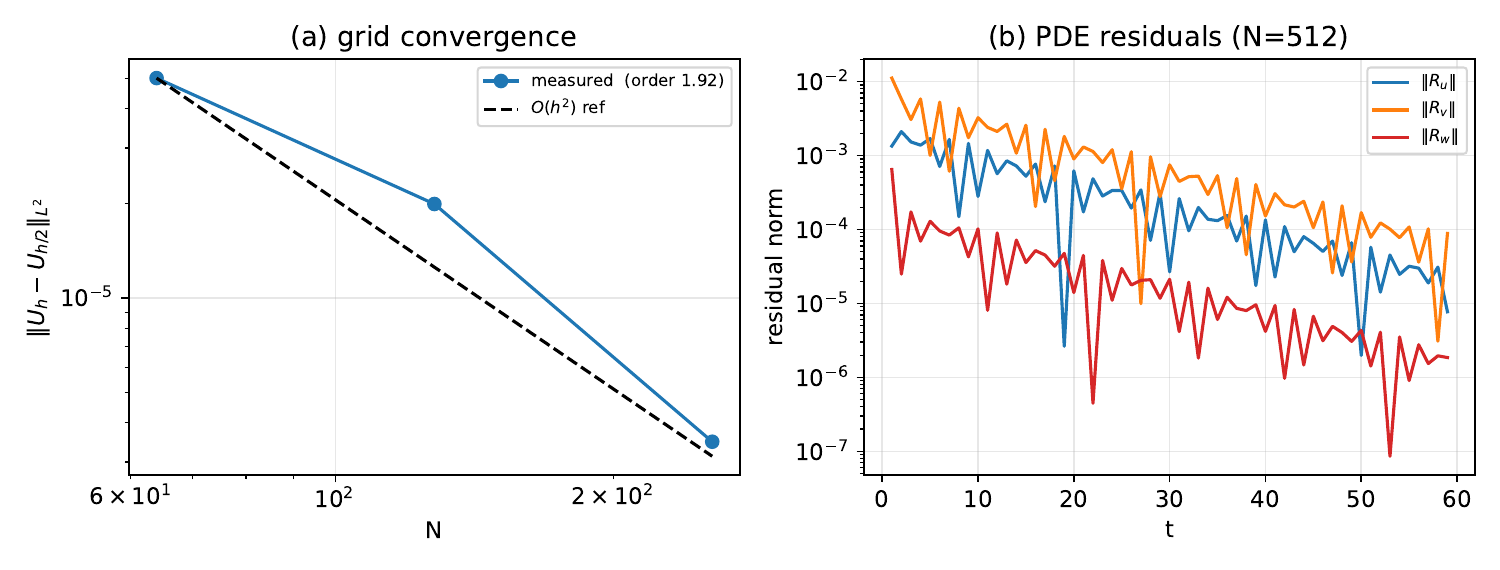}
\caption{Experiment~6. (a) Grid convergence $\|U_h-U_{h/2}\|_{L^2}$ versus $N$
(observed order $1.92$; dashed $\mathcal O(h^2)$ reference). (b) Post-hoc partial differential equation (PDE) residual
norms $\|R_u\|,\|R_v\|,\|R_w\|$ versus time on the finest grid.}
\label{fig:exp6}
\end{figure}
\clearpage

\subsection{Experiment 7: two-dimensional pattern and immune-infiltration foci}
\label{subsec:exp7-2d}

To illustrate the biologically relevant regime, we integrate the full system on
the square $\Omega=(0,10)^2$ with the two-dimensional finite-volume scheme of
Section~\ref{subsec:two-dimensional-extension}, using the CFL-limited explicit
stepper of Section~\ref{subsec:positivity-cfl} (which guarantees discrete
positivity, see Proposition~\ref{prop:positivity-cfl}). Starting from a small
perturbation of the coexistence equilibrium \eqref{eq:numeric-coexistence} at
$\xi=1.4\,\xi_c$, the solution self-organizes into the pattern of
Figure~\ref{fig:exp7-2d}: the immune density condenses into sharp aggregation
\emph{foci} located at local maxima of the chemokine field, where the tumour is
correspondingly depleted, while intervening regions remain immune-poor and
tumour-rich. The minimum of every component stays strictly positive throughout
($\min(u,v,w)\approx1.6\times10^{-3}>0$), consistent with
Proposition~\ref{prop:positivity-cfl}. This is the two-dimensional, spatially
explicit counterpart of the ``hot'' (infiltrated) versus ``cold'' (excluded)
dichotomy: chemokine-directed chemotaxis concentrates effector cells into
discrete infiltration hot-spots rather than distributing them uniformly.

\begin{figure}[htbp]
\centering
\includegraphics[width=\textwidth]{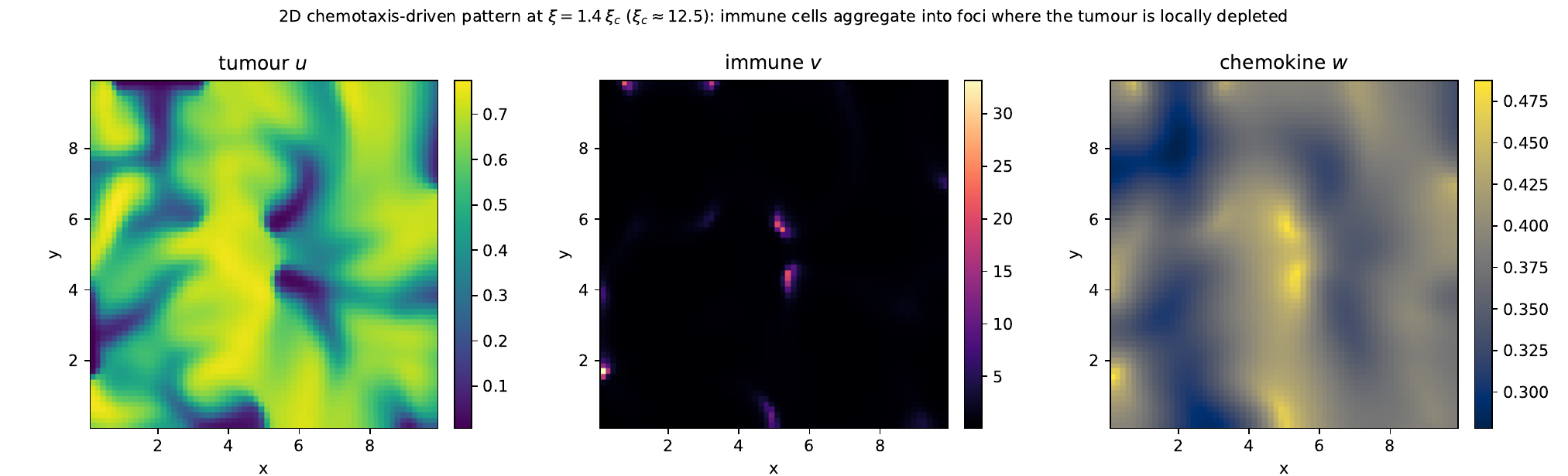}
\caption{Experiment~7. Two-dimensional chemotaxis-driven pattern at
$\xi=1.4\,\xi_c$ ($\xi_c\approx12.5$) on $\Omega=(0,10)^2$, shown at $t=80$.
Left: tumour $u$ (labyrinthine network of tumour-rich plateaus and depleted
channels); centre: immune $v$ (discrete aggregation foci, peak $\approx33$
against baseline $v^*\approx0.47$); right: chemokine $w$ (smooth field whose
maxima coincide with the immune foci). Immune effectors concentrate into
infiltration hot-spots at chemokine maxima where the tumour is locally cleared.
Discrete positivity is preserved ($\min>0$).}
\label{fig:exp7-2d}
\end{figure}
\clearpage

\subsection{Summary of verification}
\label{subsec:exp-summary}

Table~\ref{tab:verification-summary} collects the quantitative agreement between
theory and computation. Across all six experiments the simulations reproduce the
analytical thresholds and equilibria to within solver tolerance and confirm the
chemotaxis-driven origin of the spatial patterns, while the positivity,
mass-balance, and residual diagnostics certify the reliability of the numerical
results.

\begin{table}[htbp]
\centering
\caption{Quantitative verification of the theoretical results.}
\label{tab:verification-summary}
\begin{tabular}{lll}
\toprule
Experiment & Theoretical prediction & Numerical result \\
\midrule
1 (decay) & rate $1-\sigma_0/\delta=-0.6667$ & $-0.6677$ \\
1 (invasion) & rate $1-\sigma_0/\delta=+0.6667$ & $+0.6635$ \\
2 & $(u^*,v^*,w^*)=(0.5311,0.4689,0.3900)$ & error $4.4\times10^{-11}$ \\
3 & dominant mode $k^*=6$, rate $0.267$ & $k^*=6$, rate $0.273$ (corr.\ $0.980$) \\
4 & onset at $\xi_c\approx12.47$ & decay / onset / pattern across $\xi_c$ \\
5 & $\sigma_1\uparrow$ destabilizes; $\delta,d_2\uparrow$ stabilize & confirmed monotone onset curves \\
6 & convergent, positive, consistent & order $1.92$, $\min>0$, residuals $\lesssim10^{-2}$ \\
\bottomrule
\end{tabular}
\end{table}

\section{Conclusion}
\label{sec:conclusion}

We have analysed a minimal but mechanistically complete
reaction--diffusion--chemotaxis model of tumour--immune--chemokine dynamics in
which a tumour-produced chemokine both recruits immune effector cells and directs
their chemotactic migration. After nondimensionalization we established local
classical solvability, invariance of the nonnegative cone, a uniform
tumour-density bound, and global-in-time immune and chemokine mass estimates,
and we delineated precisely why these estimates fall short of global
$L^\infty$-boundedness for the full three-component chemotaxis system.

On the level of equilibria, the tumour-free state is governed by the immune
control index $\sigma_0/\delta$, with linear stability for $\sigma_0>\delta$ and
invasion for $\sigma_0<\delta$, and coexistence reduces to a single scalar
equation $F(u^*)=0$. Linearizing about coexistence and projecting onto the
Neumann eigenbasis yields the stability matrix $M_k$ in which the chemotactic
sensitivity appears only through the wavenumber-amplified coupling
$\xi v^*\mu_k$. Through the Routh--Hurwitz analysis this coupling drives a
finite-wavelength instability above a critical sensitivity $\xi_c$, which for the
parameters studied here is an oscillatory (Hopf-type) crossing rather than a
stationary one.

The numerical contribution is a conservative finite-volume scheme with upwind
chemotactic flux, discrete no-flux boundaries, and positivity, mass-balance, and
residual diagnostics. Rather than presenting simulations as illustrations, we
organized the experiments around the explicit verification of each theoretical
result. The simulations reproduced the tumour-free growth rate $1-\sigma_0/\delta$
and threshold, recovered the analytical coexistence equilibrium to solver
tolerance, matched the dispersion relation and the dominant unstable wavenumber,
exhibited the predicted homogeneous-to-heterogeneous transition across $\xi_c$,
mapped the stability boundary in three parameter planes, and passed grid
convergence and residual certification while preserving discrete positivity.

Several directions remain open. A global boundedness theory for the full system
would require additional structure, such as small chemotactic sensitivity,
stronger logistic damping, signal-dependent desensitization, or
dimension-dependent regularity estimates. The weakly nonlinear analysis of the
Hopf-type onset, the classification of the emerging spatio-temporal patterns in
two and three dimensions, and the incorporation of treatment terms
(immunotherapy, chemotherapy) are natural extensions of the present framework.

\appendix
\section{Proofs of Technical Estimates}
\label{app:technical-estimates}

This appendix collects the technical arguments supporting the local solvability,
positivity, and a priori estimates used in Section~\ref{sec:mathematical-analysis}.
The goal is not to prove a full global boundedness theorem for the
three-component chemotaxis system, but rather to record the estimates that are
sufficient for the equilibrium, stability, and numerical verification results
developed in the main text.
We consider the nondimensional system \eqref{eq:nondim-model}
posed in a bounded smooth domain $\Omega\subset\mathbb R^d$ with no-flux
boundary conditions \eqref{eq:nondim-noflux}.
The initial data are assumed to satisfy
\[
u_0,v_0,w_0\ge0
\quad\text{in }\overline\Omega .
\]

\subsection{Normal parabolicity, local solvability, and positivity}
\label{app:normal-parabolicity}
Writing the principal part of \eqref{eq:nondim-model} in divergence form,
\[
 \partial_tU=\nabla\cdot\bigl(\mathcal A(U)\nabla U\bigr)
 +\mathcal R(U),
 \qquad U=(u,v,w)^\top,
\]
gives
\[
\mathcal A(U)=
\begin{pmatrix}
d_1&0&0\\
0&d_2&-\xi v\\
0&0&d_3
\end{pmatrix}.
\]
For every $\zeta\in\mathbb R^d\setminus\{0\}$,
\[
 \sigma\!\left(|\zeta|^2\mathcal A(U)\right)
 =
 \{d_1|\zeta|^2,d_2|\zeta|^2,d_3|\zeta|^2\}.
\]
Consequently, with $d_*:=\min\{d_1,d_2,d_3\}>0$,
\[
 \min_{\lambda\in
 \sigma(|\zeta|^2\mathcal A(U))}
 \operatorname{Re}\lambda
 \ge d_*|\zeta|^2 .
\]
This is the normal-ellipticity condition. The conclusion
remains valid when two diffusivities coincide, even though the $(v,w)$ block
may then have a nontrivial Jordan form. On every set $|v|\le M$, the
coefficients and the corresponding sectoriality constants are uniform.

The boundary conditions are compatible with this normally elliptic principal
part. Indeed, $\partial_nw=0$ and
$d_2\partial_nv-\xi v\partial_nw=0$ imply $\partial_nv=0$, so the no-flux
conditions are equivalent, for classical solutions, to componentwise
homogeneous Neumann conditions. The associated boundary operator satisfies
the standard complementing condition. Moreover, $\mathcal A$ and the reaction
map are smooth on the open phase domain
\[
 \mathcal O:=\mathbb R^2\times(-1,\infty).
\]
Thus, on a smooth bounded domain and for sufficiently regular compatible
initial data, standard quasilinear parabolic theory
\cite{amann1993nonhomogeneous} yields a unique maximal classical solution.

We next prove invariance of the nonnegative cone. Fix $T<T_{\max}$. Classical
regularity implies that $u,v,w$ and $\nabla w$ are bounded on
$[0,T]\times\overline\Omega$. By continuity, it is enough initially to work on
a time interval on which $w>-1/2$. Define
\[
 u_-:=\max\{-u,0\},\qquad
 v_-:=\max\{-v,0\},\qquad
 w_-:=\max\{-w,0\}.
\]
Testing the immune equation in conservative form by $-v_-$ and using the
no-flux boundary condition gives
\[
\begin{aligned}
\frac12\frac{d}{dt}\|v_-\|_{L^2}^2
={}&-d_2\|\nabla v_-\|_{L^2}^2
+\xi\int_\Omega v_-\nabla v_-\cdot\nabla w\,dx\\
&-\int_\Omega(\delta+\beta u)v_-^2\,dx
-\int_\Omega
\left(\sigma_0+\sigma_1\frac{w}{1+w}\right)v_-\,dx .
\end{aligned}
\]
Young's inequality and boundedness of $u$ and $\nabla w$ imply
\[
 \frac12\frac{d}{dt}\|v_-\|_{L^2}^2
 +\frac{d_2}{2}\|\nabla v_-\|_{L^2}^2
 \le C_T\|v_-\|_{L^2}^2
 -\int_\Omega
 \left(\sigma_0+\sigma_1\frac{w}{1+w}\right)v_-\,dx .
\]
On $\{w<0\}$, because $w>-1/2$,
\[
 \left|\frac{w}{1+w}\right|\le 2w_-,
\]
and hence
\[
 -\left(\sigma_0+\sigma_1\frac{w}{1+w}\right)v_-
 \le 2\sigma_1w_-v_-
 \le \sigma_1\bigl(w_-^2+v_-^2\bigr).
\]
Therefore
\[
 \frac{d}{dt}\|v_-\|_{L^2}^2
 \le C_T\bigl(\|v_-\|_{L^2}^2+\|w_-\|_{L^2}^2\bigr).
\]

Testing the $u$- and $w$-equations by $-u_-$ and $-w_-$, respectively, gives
\[
 \frac{d}{dt}\|u_-\|_{L^2}^2
 \le C_T\|u_-\|_{L^2}^2
\]
and
\[
 \frac{d}{dt}\|w_-\|_{L^2}^2
 \le C_T\left(
 \|u_-\|_{L^2}^2+\|v_-\|_{L^2}^2+\|w_-\|_{L^2}^2
 \right).
\]
For the latter estimate we use boundedness of $u$ and $v$ and
\[
 (uv)_-\le |v|u_-+|u|v_-.
\]
Adding the three inequalities and applying Gronwall's lemma yields
\[
 \|u_-(t)\|_{L^2}^2+
 \|v_-(t)\|_{L^2}^2+
 \|w_-(t)\|_{L^2}^2=0
\]
because all three negative parts vanish initially. A continuation argument
then extends the conclusion to every $T<T_{\max}$. Thus
\[
 u(t,x)\ge0,\qquad v(t,x)\ge0,\qquad w(t,x)\ge0
\]
throughout the classical solution's interval of existence.

\subsection{Tumour-density bound}
\label{app:tumour-bound}
We prove Proposition~\ref{prop:tumour-density-bound}. Since $v\ge0$, the first
equation of \eqref{eq:nondim-model} satisfies
\[
u_t
=
d_1\Delta u+u(1-u-v)
\le
d_1\Delta u+u(1-u).
\]
Let $y(t)$ solve the scalar logistic equation
\[
y'=y(1-y),
\qquad
y(0)=\|u_0\|_{L^\infty(\Omega)}.
\]
The comparison principle under homogeneous Neumann boundary conditions gives
\[
0\le u(t,x)\le y(t)
\quad
\text{for }0<t<T_{\max},\ x\in\overline\Omega.
\]
The explicit solution of the logistic equation satisfies
\[
y(t)\le \max\{1,y(0)\}.
\]
Therefore
\begin{equation}\label{eq:app-u-bound}
0\le u(t,x)\le M_u
:=
\max\{1,\|u_0\|_{L^\infty(\Omega)}\}
\quad
\text{for }0<t<T_{\max}.
\end{equation}

\subsection{Immune mass estimate}
\label{app:v-mass}
We prove the first estimate in Proposition~\ref{prop:mass-estimates}. Integrating
the $v$-equation in \eqref{eq:nondim-model} over $\Omega$ gives
\[
\frac{d}{dt}\int_\Omega v\,dx
=
d_2\int_\Omega \Delta v\,dx
-\xi\int_\Omega \nabla\cdot(v\nabla w)\,dx
+
\int_\Omega
\left(
\sigma_0+\sigma_1\frac{w}{1+w}-\delta v-\beta uv
\right)dx.
\]
The diffusion and chemotaxis integrals vanish because of the no-flux condition:
\[
\int_\Omega \nabla\cdot(d_2\nabla v-\xi v\nabla w)\,dx
=
\int_{\partial\Omega}(d_2\nabla v-\xi v\nabla w)\cdot n\,dS
=
0.
\]
Since $u,v,w\ge0$ and
\[
0\le \frac{w}{1+w}\le1,
\]
we obtain
\begin{equation}\label{eq:app-v-differential}
\frac{d}{dt}\int_\Omega v\,dx
\le
|\Omega|(\sigma_0+\sigma_1)-\delta\int_\Omega v\,dx.
\end{equation}
Solving the scalar differential inequality yields
\begin{equation}\label{eq:app-v-bound}
\int_\Omega v(t,x)\,dx
\le
e^{-\delta t}\int_\Omega v_0(x)\,dx
+
\frac{|\Omega|(\sigma_0+\sigma_1)}{\delta}
\left(1-e^{-\delta t}\right).
\end{equation}
Consequently,
\begin{equation}\label{eq:app-v-uniform-bound}
\sup_{0<t<T_{\max}}\int_\Omega v(t,x)\,dx
\le
C_v,
\end{equation}
where
\[
C_v
:=
\max\left\{
\int_\Omega v_0(x)\,dx,\,
\frac{|\Omega|(\sigma_0+\sigma_1)}{\delta}
\right\}.
\]

\subsection{Chemokine mass estimate}
\label{app:w-mass}
We now prove the chemokine mass estimate. Integrating the $w$-equation in
\eqref{eq:nondim-model} over $\Omega$ and using the no-flux boundary condition
gives
\[
\frac{d}{dt}\int_\Omega w\,dx
=
\int_\Omega
\left(
\alpha u+\gamma uv-\ell w
\right)dx.
\]
Using the tumour-density bound \eqref{eq:app-u-bound}, we obtain
\[
\int_\Omega \alpha u\,dx
\le
\alpha M_u|\Omega|,
\qquad
\int_\Omega \gamma uv\,dx
\le
\gamma M_u\int_\Omega v\,dx.
\]
Therefore
\begin{equation}\label{eq:app-w-differential}
\frac{d}{dt}\int_\Omega w\,dx
\le
\alpha M_u|\Omega|
+
\gamma M_u\int_\Omega v\,dx
-
\ell\int_\Omega w\,dx.
\end{equation}
By \eqref{eq:app-v-uniform-bound},
\[
\int_\Omega v(t,x)\,dx\le C_v
\quad\text{for all }t<T_{\max}.
\]
Thus
\[
\frac{d}{dt}\int_\Omega w\,dx
\le
\alpha M_u|\Omega|+\gamma M_u C_v
-\ell\int_\Omega w\,dx.
\]
Solving this scalar differential inequality gives
\begin{equation}\label{eq:app-w-bound}
\int_\Omega w(t,x)\,dx
\le
e^{-\ell t}\int_\Omega w_0(x)\,dx
+
\frac{\alpha M_u|\Omega|+\gamma M_u C_v}{\ell}
\left(1-e^{-\ell t}\right).
\end{equation}
Consequently,
\begin{equation}\label{eq:app-w-uniform-bound}
\sup_{0<t<T_{\max}}\int_\Omega w(t,x)\,dx
\le
C_w,
\end{equation}
where
\[
C_w
:=
\max\left\{
\int_\Omega w_0(x)\,dx,\,
\frac{\alpha M_u|\Omega|+\gamma M_u C_v}{\ell}
\right\}.
\]

\subsection{Why these estimates do not imply full global boundedness}
\label{app:no-global-linfty}
The estimates derived above imply
\[
\|u(t)\|_{L^\infty(\Omega)}\le M_u,
\qquad
\|v(t)\|_{L^1(\Omega)}\le C_v,
\qquad
\|w(t)\|_{L^1(\Omega)}\le C_w
\]
on the interval of classical existence. These estimates are sufficient for the
basic biological interpretation of the model and for the equilibrium calculations
in the main text. However, they do not by themselves imply
\[
\sup_{0<t<T_{\max}}
\left(
\|v(t)\|_{L^\infty(\Omega)}
+
\|w(t)\|_{L^\infty(\Omega)}
\right)<\infty.
\]
The reason is that the immune equation contains the chemotactic aggregation
term
\[
-\xi\nabla\cdot(v\nabla w),
\]
which can transfer spatial gradients of the chemokine field into concentration
of the immune density. An $L^1$ bound on $v$ and an $L^1$ bound on $w$ do not
control the gradients $\nabla w$ or the pointwise size of $v$.
Thus a full global boundedness theorem would require additional structure, such
as one or more of the following: small chemotactic sensitivity, strong logistic
damping, signal-dependent chemotactic desensitization, entropy or Lyapunov
estimates, dimension-dependent regularity arguments, or stronger a priori control
of $\nabla w$.
For this reason, the main text states local classical well-posedness, positivity,
tumour-density boundedness, and mass estimates, but does not claim global
$L^\infty$-boundedness for the full three-component chemotaxis system in
arbitrary dimension.

\section{Details of the Routh--Hurwitz Calculation}
\label{app:routh-hurwitz}
This appendix gives the algebra leading to the mode-wise stability criterion used
in Sections~\ref{sec:linear-stability} and~\ref{sec:numerical-experiments}. We
linearize the nondimensional system \eqref{eq:nondim-noflux} about 
a spatially homogeneous coexistence equilibrium
\[
(u^*,v^*,w^*),\qquad
u^*>0,\quad v^*>0,\quad w^*>0,
\]
where
\begin{equation}\label{eq:app-rh-coexistence}
v^*=1-u^*,
\qquad
w^*=\frac{u^*[\alpha+\gamma(1-u^*)]}{\ell}.
\end{equation}

\subsection{Linearization about coexistence}
\label{app:rh-linearization}
Let
\[
u=u^*+\widetilde u,\qquad
v=v^*+\widetilde v,\qquad
w=w^*+\widetilde w .
\]
Keeping only first-order terms gives
\begin{equation}\label{eq:app-rh-linearized-system}
\begin{aligned}
\widetilde u_t
&=
d_1\Delta\widetilde u
+
(1-2u^*-v^*)\widetilde u
-u^*\widetilde v,\\
\widetilde v_t
&=
d_2\Delta\widetilde v
-\xi v^*\Delta\widetilde w
-\beta v^*\widetilde u
-(\delta+\beta u^*)\widetilde v
+\frac{\sigma_1}{(1+w^*)^2}\widetilde w,\\
\widetilde w_t
&=
d_3\Delta\widetilde w
+
(\alpha+\gamma v^*)\widetilde u
+\gamma u^*\widetilde v
-\ell\widetilde w .
\end{aligned}
\end{equation}
Since \(v^*=1-u^*\), we have $1-2u^*-v^*=-u^*$, so the first equation becomes
\[
\widetilde u_t
=
d_1\Delta\widetilde u-u^*\widetilde u-u^*\widetilde v .
\]
The sign of the chemotaxis term is important. Since the chemotactic contribution
is $-\xi\nabla\cdot(v\nabla w)$, its linearization at the homogeneous equilibrium
is
\[
-\xi\nabla\cdot(v^*\nabla\widetilde w)
=
-\xi v^*\Delta\widetilde w .
\]
After projection onto a Neumann eigenmode, this term becomes a positive
\((v,w)\)-coupling, as shown below.

\subsection{Projection onto Neumann modes}
\label{app:rh-neumann-modes}
Let \(\{\phi_k\}_{k=0}^{\infty}\) be the Neumann Laplacian eigenfunctions,
\[
-\Delta\phi_k=\mu_k\phi_k,\qquad
\partial_n\phi_k=0\quad\text{on }\partial\Omega,
\qquad
0=\mu_0<\mu_1\le\mu_2\le\cdots .
\]
For a single mode, write
\[
\widetilde u(t,x)=U_k(t)\phi_k(x),\qquad
\widetilde v(t,x)=V_k(t)\phi_k(x),\qquad
\widetilde w(t,x)=W_k(t)\phi_k(x).
\]
Since \(\Delta\phi_k=-\mu_k\phi_k\), the modal amplitudes satisfy
\[
\frac{d}{dt}
\begin{pmatrix}
U_k\\ V_k\\ W_k
\end{pmatrix}
=
M_k
\begin{pmatrix}
U_k\\ V_k\\ W_k
\end{pmatrix},
\]
where
\begin{equation}\label{eq:app-Mk}
M_k=
\begin{pmatrix}
-u^*-d_1\mu_k & -u^* & 0\\[2mm]
-\beta v^* & -(\delta+\beta u^*)-d_2\mu_k &
\dfrac{\sigma_1}{(1+w^*)^2}+\xi v^*\mu_k\\[3mm]
\alpha+\gamma v^* & \gamma u^* & -\ell-d_3\mu_k
\end{pmatrix}.
\end{equation}
The positive contribution \(+\xi v^*\mu_k\) in the \((2,3)\)-entry follows from
\[
-\xi v^*\Delta\widetilde w
=
-\xi v^*(-\mu_k W_k\phi_k)
=
\xi v^*\mu_k W_k\phi_k .
\]
Thus chemotaxis has no effect on the homogeneous mode through this term because
\(\mu_0=0\), but it can destabilize finite spatial modes with \(\mu_k>0\). This is
the matrix \eqref{eq:Mk-matrix} used in the main text.

\subsection{Characteristic polynomial}
\label{app:rh-characteristic-polynomial}
For compactness define
\begin{equation}\label{eq:app-rh-notation}
A_k=u^*+d_1\mu_k,\qquad
B_k=\delta+\beta u^*+d_2\mu_k,\qquad
C_k=\ell+d_3\mu_k,
\end{equation}
and
\begin{equation}\label{eq:app-rh-notation-2}
p=u^*,\qquad
q=\beta v^*,\qquad
r_k=\frac{\sigma_1}{(1+w^*)^2}+\xi v^*\mu_k,\qquad
s=\alpha+\gamma v^*,\qquad
t=\gamma u^*.
\end{equation}
Then \eqref{eq:app-Mk} can be written as
\begin{equation}\label{eq:app-Mk-compact}
M_k=
\begin{pmatrix}
-A_k & -p & 0\\
-q & -B_k & r_k\\
s & t & -C_k
\end{pmatrix}.
\end{equation}
The characteristic polynomial is \(P_k(\lambda)=\det(\lambda I-M_k)\). Since
\[
\lambda I-M_k
=
\begin{pmatrix}
\lambda+A_k & p & 0\\
q & \lambda+B_k & -r_k\\
-s & -t & \lambda+C_k
\end{pmatrix},
\]
expansion along the first row gives
\[
P_k(\lambda)
=
(\lambda+A_k)
\left[
(\lambda+B_k)(\lambda+C_k)-r_k t
\right]
-
p\left[
q(\lambda+C_k)-r_k s
\right].
\]
Therefore
\[
P_k(\lambda)
=
\lambda^3+a_1(k)\lambda^2+a_2(k)\lambda+a_3(k),
\]
where
\begin{equation}\label{eq:app-a1}
a_1(k)=A_k+B_k+C_k,
\end{equation}
\begin{equation}\label{eq:app-a2}
a_2(k)=A_kB_k+A_kC_k+B_kC_k-r_k t-pq,
\end{equation}
\begin{equation}\label{eq:app-a3}
a_3(k)=A_kB_kC_k-A_k r_k t-pq C_k+p r_k s.
\end{equation}
Equivalently,
\begin{equation}\label{eq:app-a3-alternative}
a_3(k)
=
A_kB_kC_k-C_kpq+r_k(ps-A_kt).
\end{equation}

\subsection{Routh--Hurwitz conditions}
\label{app:rh-conditions}
The mode \(k\) is linearly stable if and only if all roots of
\(P_k(\lambda)=\lambda^3+a_1(k)\lambda^2+a_2(k)\lambda+a_3(k)\) lie in the open
left half-plane. By the Routh--Hurwitz criterion for a monic cubic, this holds if
and only if
\begin{equation}\label{eq:app-rh-conditions}
a_1(k)>0,\qquad
a_2(k)>0,\qquad
a_3(k)>0,\qquad
a_1(k)a_2(k)>a_3(k).
\end{equation}
Since \(A_k,B_k,C_k>0\), we always have \(a_1(k)>0\). The remaining inequalities
encode the competition between diffusive damping, reaction damping, chemokine
recruitment, and chemotactic amplification.

\subsection{Chemotactic contribution to the constant term}
\label{app:rh-chemotaxis-constant-term}
The dependence of \(a_3(k)\) on the chemotactic sensitivity \(\xi\) is especially
important. Since \(r_k=\sigma_1/(1+w^*)^2+\xi v^*\mu_k\), we can rewrite
\eqref{eq:app-a3-alternative} as
\begin{equation}\label{eq:app-a3-xi}
a_3(k;\xi)
=
\Lambda_k+\xi v^*\mu_k(ps-A_kt),
\end{equation}
where
\begin{equation}\label{eq:app-Lambda-k}
\Lambda_k
=
A_kB_kC_k-C_kpq
+
\frac{\sigma_1}{(1+w^*)^2}(ps-A_kt).
\end{equation}
Equivalently,
\begin{equation}\label{eq:app-a3-xi-negative-form}
a_3(k;\xi)
=
\Lambda_k-\xi v^*\mu_k(A_kt-ps).
\end{equation}
Thus, if
\begin{equation}\label{eq:app-destabilizing-condition}
A_kt-ps>0,
\end{equation}
then increasing \(\xi\) decreases \(a_3(k;\xi)\). In this case, if
\(\Lambda_k>0\), the \(k\)-th mode violates the Routh--Hurwitz condition
\(a_3(k)>0\) whenever
\begin{equation}\label{eq:app-xi-k-star}
\xi>\xi_k^*
:=
\frac{\Lambda_k}{v^*\mu_k(A_kt-ps)} .
\end{equation}
If \(\Lambda_k\le0\), then \(a_3(k;\xi)\le0\) already at \(\xi=0\), so the mode
is unstable before chemotaxis is increased. Hence \eqref{eq:app-xi-k-star} is a
useful chemotactic threshold only in the regime \(\Lambda_k>0,\ A_kt-ps>0\). The
first instability generated through this mechanism occurs at
\begin{equation}\label{eq:app-xi-critical}
\xi_c
=
\min_{\substack{k\ge1\\ \Lambda_k>0,\ A_kt-ps>0}}
\xi_k^* .
\end{equation}
The restriction \(k\ge1\) is necessary because the homogeneous mode has
\(\mu_0=0\), and the chemotactic contribution \(\xi v^*\mu_k\) vanishes for
\(k=0\).

\subsection{Alternative instability through the determinant condition}
\label{app:rh-other-instabilities}
The condition \(a_3(k)>0\) is only one of the Routh--Hurwitz inequalities.
A mode may also lose stability through \(a_2(k)=0\) or through
\(a_1(k)a_2(k)-a_3(k)=0\). The latter corresponds to a Hopf-type crossing for the
modal ODE system when the remaining Routh--Hurwitz inequalities hold. Therefore,
in numerical stability maps it is safest to compute
\[
\omega_k(\xi)
=
\max_j\operatorname{Re}\lambda_j(M_k)
\]
directly and define the onset by \(\max_{k\ge1}\omega_k(\xi)=0\). The explicit
threshold \eqref{eq:app-xi-k-star} should be interpreted as an analytical
sufficient destabilization criterion through the constant term \(a_3(k)\), not as
the only possible instability mechanism. Indeed, for the parameter set of
Section~\ref{sec:numerical-experiments} one has \(A_kt-ps<0\) in the relevant
wavenumber band, so the constant-term mechanism is inactive and the onset is the
Hopf-type crossing.

\bibliographystyle{unsrt}
\bibliography{reference}

@article{anderson2008integrative,
  title={Integrative mathematical oncology},
  author={Anderson, Alexander RA and Quaranta, Vito},
  journal={Nature Reviews Cancer},
  volume={8},
  number={3},
  pages={227--234},
  year={2008},
  publisher={Nature Publishing Group UK London}
}

@article{altrock2015mathematics,
  title={The mathematics of cancer: integrating quantitative models},
  author={Altrock, Philipp M and Liu, Lin L and Michor, Franziska},
  journal={Nature Reviews Cancer},
  volume={15},
  number={12},
  pages={730--745},
  year={2015},
  publisher={Nature Publishing Group UK London}
}

@article{rockne20192019,
  title={The 2019 mathematical oncology roadmap},
  author={Rockne, Russell C and Hawkins-Daarud, Andrea and Swanson, Kristin R and Sluka, James P and Glazier, James A and Macklin, Paul and Hormuth, David A and Jarrett, Angela M and Lima, Ernesto ABF and Tinsley Oden, J and others},
  journal={Physical biology},
  volume={16},
  number={4},
  pages={041005},
  year={2019},
  publisher={IOP Publishing}
}

@article{metzcar2024review,
  title={A review of mechanistic learning in mathematical oncology},
  author={Metzcar, John and Jutzeler, Catherine R and Macklin, Paul and K{\"o}hn-Luque, Alvaro and Br{\"u}ningk, Sarah C},
  journal={Frontiers in Immunology},
  volume={15},
  pages={1363144},
  year={2024},
  publisher={Frontiers Media SA}
}

@article{roose2007mathematical,
  title={Mathematical models of avascular tumor growth},
  author={Roose, Tiina and Chapman, S Jonathan and Maini, Philip K},
  journal={SIAM review},
  volume={49},
  number={2},
  pages={179--208},
  year={2007},
  publisher={SIAM}
}

@article{kuznetsov1994nonlinear,
  title={Nonlinear dynamics of immunogenic tumors: parameter estimation and global bifurcation analysis},
  author={Kuznetsov, Vladimir A and Makalkin, Iliya A and Taylor, Mark A and Perelson, Alan S},
  journal={Bulletin of mathematical biology},
  volume={56},
  number={2},
  pages={295--321},
  year={1994},
  publisher={Elsevier}
}

@article{kirschner1998modeling,
  title={Modeling immunotherapy of the tumor--immune interaction},
  author={Kirschner, Denise and Panetta, John Carl},
  journal={Journal of mathematical biology},
  volume={37},
  number={3},
  pages={235--252},
  year={1998},
  publisher={Springer}
}

@article{de2001mathematical,
  title={A mathematical tumor model with immune resistance and drug therapy: an optimal control approach},
  author={De Pillis, Lisette G and Radunskaya, Ami},
  journal={Computational and Mathematical Methods in Medicine},
  volume={3},
  number={2},
  pages={79--100},
  year={2001},
  publisher={Taylor \& Francis}
}

@article{eftimie2011interactions,
  title={Interactions between the immune system and cancer: a brief review of non-spatial mathematical models},
  author={Eftimie, Raluca and Bramson, Jonathan L and Earn, David JD},
  journal={Bulletin of mathematical biology},
  volume={73},
  number={1},
  pages={2--32},
  year={2011},
  publisher={Springer}
}

@article{matzavinos2004mathematical,
  title={Mathematical modelling of the spatio-temporal response of cytotoxic T-lymphocytes to a solid tumour},
  author={Matzavinos, Anastasios and Chaplain, Mark AJ and Kuznetsov, Vladimir A},
  journal={Mathematical Medicine and Biology},
  volume={21},
  number={1},
  pages={1--34},
  year={2004},
  publisher={Oxford University Press}
}

@article{anderson1998continuous,
  title={Continuous and discrete mathematical models of tumor-induced angiogenesis},
  author={Anderson, Alexander RA and Chaplain, Mark AJ},
  journal={Bulletin of mathematical biology},
  volume={60},
  number={5},
  pages={857--899},
  year={1998},
  publisher={Elsevier}
}

@article{ghaffarizadeh2018physicell,
  title={PhysiCell: An open source physics-based cell simulator for 3-D multicellular systems},
  author={Ghaffarizadeh, Ahmadreza and Heiland, Randy and Friedman, Samuel H and Mumenthaler, Shannon M and Macklin, Paul},
  journal={PLoS computational biology},
  volume={14},
  number={2},
  pages={e1005991},
  year={2018},
  publisher={Public Library of Science San Francisco, CA USA}
}

@article{nagarsheth2017chemokines,
  title={Chemokines in the cancer microenvironment and their relevance in cancer immunotherapy},
  author={Nagarsheth, Nisha and Wicha, Max S and Zou, Weiping},
  journal={Nature Reviews Immunology},
  volume={17},
  number={9},
  pages={559--572},
  year={2017},
  publisher={Nature Publishing Group UK London}
}

@article{galon2006type,
  title={Type, density, and location of immune cells within human colorectal tumors predict clinical outcome},
  author={Galon, J{\'e}r{\^o}me and Costes, Anne and Sanchez-Cabo, Fatima and Kirilovsky, Amos and Mlecnik, Bernhard and Lagorce-Pag{\`e}s, Christine and Tosolini, Marie and Camus, Matthieu and Berger, Anne and Wind, Philippe and others},
  journal={Science},
  volume={313},
  number={5795},
  pages={1960--1964},
  year={2006},
  publisher={American Association for the Advancement of Science}
}

@article{galon2019approaches,
  title={Approaches to treat immune hot, altered and cold tumours with combination immunotherapies},
  author={Galon, J{\'e}r{\^o}me and Bruni, Daniela},
  journal={Nature reviews Drug discovery},
  volume={18},
  number={3},
  pages={197--218},
  year={2019},
  publisher={Nature Publishing Group UK London}
}

@article{joyce2015t,
  title={T cell exclusion, immune privilege, and the tumor microenvironment},
  author={Joyce, Johanna A and Fearon, Douglas T},
  journal={Science},
  volume={348},
  number={6230},
  pages={74--80},
  year={2015},
  publisher={American Association for the Advancement of Science}
}

@article{keller1970initiation,
  title={Initiation of slime mold aggregation viewed as an instability},
  author={Keller, Evelyn F and Segel, Lee A},
  journal={Journal of theoretical biology},
  volume={26},
  number={3},
  pages={399--415},
  year={1970},
  publisher={Elsevier}
}

@article{hillen2009user,
  title={A user’s guide to PDE models for chemotaxis},
  author={Hillen, Thomas and Painter, Kevin J},
  journal={Journal of mathematical biology},
  volume={58},
  number={1},
  pages={183--217},
  year={2009},
  publisher={Springer}
}

@article{winkler2010aggregation,
  title={Aggregation vs. global diffusive behavior in the higher-dimensional Keller--Segel model},
  author={Winkler, Michael},
  journal={Journal of Differential Equations},
  volume={248},
  number={12},
  pages={2889--2905},
  year={2010},
  publisher={Elsevier}
}

@article{tao2012boundedness,
  title={Boundedness in a quasilinear parabolic--parabolic Keller--Segel system with subcritical sensitivity},
  author={Tao, Youshan and Winkler, Michael},
  journal={Journal of Differential Equations},
  volume={252},
  number={1},
  pages={692--715},
  year={2012},
  publisher={Elsevier}
}

@incollection{amann1993nonhomogeneous,
  title={Nonhomogeneous linear and quasilinear elliptic and parabolic boundary value problems},
  author={Amann, Herbert},
  booktitle={Function spaces, differential operators and nonlinear analysis},
  pages={9--126},
  year={1993},
  publisher={Springer}
}

@book{murray2003mathematical,
  title={Mathematical biology: II: spatial models and biomedical applications},
  author={Murray, James Dickson and Murray, James Dickson},
  volume={18},
  year={2003},
  publisher={Springer}
}

@article{filbet2006finite,
  title={A finite volume scheme for the Patlak--Keller--Segel chemotaxis model},
  author={Filbet, Francis},
  journal={Numerische Mathematik},
  volume={104},
  number={4},
  pages={457--488},
  year={2006},
  publisher={Springer}
}

@article{chertock2008second,
  title={A second-order positivity preserving central-upwind scheme for chemotaxis and haptotaxis models},
  author={Chertock, Alina and Kurganov, Alexander},
  journal={Numerische Mathematik},
  volume={111},
  number={2},
  pages={169--205},
  year={2008},
  publisher={Springer}
}

@article{saito2007conservative,
  title={Conservative upwind finite-element method for a simplified Keller--Segel system modelling chemotaxis},
  author={Saito, Norikazu},
  journal={IMA journal of numerical analysis},
  volume={27},
  number={2},
  pages={332--365},
  year={2007},
  publisher={OUP}
}

@article{eymard2000finite,
  title={Finite volume methods},
  author={Eymard, Robert and Gallou{\"e}t, Thierry and Herbin, Rapha{\`e}le},
  journal={Handbook of numerical analysis},
  volume={7},
  pages={713--1018},
  year={2000},
  publisher={Elsevier}
}

@article{beeghly2023measuring,
  title={Measuring and modelling tumour heterogeneity across scales},
  author={Beeghly, Garrett F and Shimpi, Adrian A and Riter, Robert N and Fischbach, Claudia},
  journal={Nature Reviews Bioengineering},
  volume={1},
  number={10},
  pages={712--730},
  year={2023},
  publisher={Nature Publishing Group UK London}
}

@article{desoyer2025computational,
  title={Computational frameworks for modelling cancer across scales},
  author={Desoyer, Celine and Loibner, Daniel and Brislinger, Dagmar and Baumgartner, Christian},
  journal={Clinical and Translational Discovery},
  volume={5},
  number={6},
  pages={e70093},
  year={2025},
  publisher={Wiley Online Library}
}

@article{tiwari2025molecular,
  title={Molecular basis of cancer},
  author={Tiwari, Lalima and Kujan, Omar},
  journal={Pathological Basis of Oral and Maxillofacial Diseases},
  pages={415--428},
  year={2025},
  publisher={Wiley Online Library}
}

@article{kaminska2015role,
  title={The role of the cell--cell interactions in cancer progression},
  author={Kami{\'n}ska, Katarzyna and Szczylik, Cezary and Bielecka, Zofia F and Bartnik, Ewa and Porta, Camillo and Lian, Fei and Czarnecka, Anna M},
  journal={Journal of cellular and molecular medicine},
  volume={19},
  number={2},
  pages={283--296},
  year={2015},
  publisher={Wiley Online Library}
}

@article{joshi2026dynamic,
  title={The Dynamic Tumor ECM: Biophysical Cues, Cellular Crosstalk, and Disease Progression},
  author={Joshi, Omkar and Hamidi, Hellyeh and Mathieu, Mathilde and Ivaska, Johanna},
  journal={Current Opinion in Biomedical Engineering},
  pages={100652},
  year={2026},
  publisher={Elsevier}
}

@article{he2022extracellular,
  title={Extracellular matrix in cancer progression and therapy},
  author={He, Xiuxiu and Lee, Byoungkoo and Jiang, Yi},
  journal={Medical Review},
  volume={2},
  number={2},
  pages={125--139},
  year={2022},
  publisher={De Gruyter}
}

@article{giuliani2025immune,
  title={Immune and non-immune cell fencing of tumor cells is a widespread and functionally relevant spatial pattern in solid cancers},
  author={Giuliani, Giuseppe and Chavan, Dhruv and Gaddam, Rushil and Stewart, William C and Li, Zihai and Jayaprakash, Ciriyam and Das, Jayajit},
  journal={Computational and Structural Biotechnology Journal},
  year={2025},
  publisher={Elsevier}
}

@article{biswas2022inference,
  title={Inference on spatial heterogeneity in tumor microenvironment using spatial transcriptomics data},
  author={Biswas, Antara and Ghaddar, Bassel and Riedlinger, Gregory and De, Subhajyoti},
  journal={Computational and systems oncology},
  volume={2},
  number={3},
  pages={e21043},
  year={2022},
  publisher={Wiley Online Library}
}

@article{kondo2010reaction,
  title={Reaction-diffusion model as a framework for understanding biological pattern formation},
  author={Kondo, Shigeru and Miura, Takashi},
  journal={science},
  volume={329},
  number={5999},
  pages={1616--1620},
  year={2010},
  publisher={American Association for the Advancement of Science}
}

@article{mempel2024chemokines,
  title={How chemokines organize the tumour microenvironment},
  author={Mempel, Thorsten R and Lill, Julia K and Altenburger, Lukas M},
  journal={Nature Reviews Cancer},
  volume={24},
  number={1},
  pages={28--50},
  year={2024},
  publisher={Nature Publishing Group UK London}
}

@article{tao2024global,
  title={Global smooth solutions in a chemotaxis system modeling immune response to a solid tumor},
  author={Tao, Youshan and Winkler, Michael},
  journal={Proceedings of the American Mathematical Society},
  volume={152},
  number={10},
  pages={4325--4341},
  year={2024}
}

@article{li2026global,
  title={Global boundedness of an ND chemotactic tumor immune evasion system},
  author={Li, Kaiqiang and Li, Yingying},
  journal={Discrete and Continuous Dynamical Systems-B},
  pages={0--0},
  year={2026},
  publisher={Discrete and Continuous Dynamical Systems-B}
}

@article{ai2015reaction,
  title={Reaction, diffusion and chemotaxis in wave propagation},
  author={Ai, Shangbing and Huang, Wenzhang and Wang, Zhi-An},
  journal={Discrete Contin. Dyn. Syst. Ser. B},
  volume={20},
  number={1},
  pages={1--21},
  year={2015}
}

@book{ke2022analysis,
  title={Analysis of reaction-diffusion models with the taxis mechanism},
  author={Ke, Yuanyuan and Li, Jing and Wang, Yifu},
  year={2022},
  publisher={Springer}
}

@article{kiselev2022chemotaxis,
  title={Chemotaxis and reactions in biology},
  author={Kiselev, Alexander and Nazarov, Fedor and Ryzhik, Lenya and Yao, Yao},
  journal={Journal of the European Mathematical Society},
  volume={25},
  number={7},
  pages={2641--2696},
  year={2022}
}

@article{mcdonald2023computational,
  title={Computational approaches to modelling and optimizing cancer treatment},
  author={McDonald, Thomas O and Cheng, Yu-Chen and Graser, Christopher and Nicol, Phillip B and Temko, Daniel and Michor, Franziska},
  journal={Nature Reviews Bioengineering},
  volume={1},
  number={10},
  pages={695--711},
  year={2023},
  publisher={Nature Publishing Group UK London}
}

@article{gopukumar2026multiscale,
  title={Multiscale predictive cellular modeling: integrating hypothesis grammars, digital twins, and multi-omics for In silico oncology and precision theranostics},
  author={Gopukumar, ST and Dwivedi, Dyumn and Rahamathulla, Mohamed and Ahmed, Mohammed Muqtader and Soni, Tanveen Kaur and Natarajan, Praveen Ganesh and Ramkanth, S and Mitra, Arpita and Das, Uddalak},
  journal={Functional \& Integrative Genomics},
  volume={26},
  number={1},
  pages={113},
  year={2026},
  publisher={Springer}
}

@article{lan2025shallow,
  title={From shallow to deep: the evolution of machine learning and mechanistic model integration in cancer research},
  author={Lan, Yunduo and Shin, Sung-Young and Nguyen, Lan K},
  journal={Current Opinion in Systems Biology},
  volume={40},
  pages={100541},
  year={2025},
  publisher={Elsevier}
}

@article{de2025radiation,
  title={Radiation-induced lymphopenia: From mathematical modeling towards mechanistic learning},
  author={de Kermenguy, Fran{\c{c}}ois and Morel, Daphn{\'e} and El-Aichi, Mohammed and Barbolosi, Dominique and Deutsch, Eric and Robert, Charlotte},
  journal={International Journal of Radiation Oncology* Biology* Physics},
  year={2025},
  publisher={Elsevier}
}

@inproceedings{laslo2025mechanistic,
  title={Mechanistic learning with guided diffusion models to predict spatio-temporal brain tumor growth},
  author={Laslo, Daria and Georgiou, Efthymios and Linguraru, Marius George and Rauschecker, Andreas M and M{\"u}ller, Sabine and Jutzeler, Catherine R and Br{\"u}ningk, Sarah},
  booktitle={International Workshop on Learning with Longitudinal Medical Images and Data},
  pages={68--79},
  year={2025},
  organization={Springer}
}

@article{mehdizadeh2023targeting,
  title={Targeting myeloid-derived suppressor cells in combination with tumor cell vaccination predicts anti-tumor immunity and breast cancer dormancy: An in silico experiment},
  author={Mehdizadeh, Reza and Shariatpanahi, Seyed Peyman and Goliaei, Bahram and R{\"u}egg, Curzio},
  journal={Scientific reports},
  volume={13},
  number={1},
  pages={5875},
  year={2023},
  publisher={Nature Publishing Group UK London}
}

@article{barrera2025impact,
  title={Impact of mild hyperthermia on tumor-immune dynamics explored through mathematical modeling},
  author={Barrera-Le{\'o}n, Andr{\'e}s Sebasti{\'a}n and Rodr{\'\i}guez-Qui{\~n}ones, Leoncio and De Mendoza, Adriana M},
  journal={Scientific Reports},
  year={2025},
  publisher={Nature Publishing Group UK London}
}

@article{hadjigeorgiou2024hybrid,
  title={Hybrid model of tumor growth, angiogenesis and immune response yields strategies to improve antiangiogenic therapy},
  author={Hadjigeorgiou, Andreas G and Stylianopoulos, Triantafyllos},
  journal={npj Biological Physics and Mechanics},
  volume={1},
  number={1},
  pages={4},
  year={2024},
  publisher={Nature Publishing Group}
}

@article{agosti2023image,
  title={An image-informed Cahn--Hilliard Keller--Segel multiphase field model for tumor growth with angiogenesis},
  author={Agosti, Abramo and Lucifero, A Giotta and Luzzi, Sabino},
  journal={Applied Mathematics and Computation},
  volume={445},
  pages={127834},
  year={2023},
  publisher={Elsevier}
}

@article{ferreira2002reaction,
  title={Reaction-diffusion model for the growth of avascular tumor},
  author={Ferreira Jr, SC and Martins, Marcelo Lobato and Vilela, MJ},
  journal={Physical Review E},
  volume={65},
  number={2},
  pages={021907},
  year={2002},
  publisher={APS}
}

@article{chaplain1996avascular,
  title={Avascular growth, angiogenesis and vascular growth in solid tumours: The mathematical modelling of the stages of tumour development},
  author={Chaplain, Mark AJ},
  journal={Mathematical and computer modelling},
  volume={23},
  number={6},
  pages={47--87},
  year={1996},
  publisher={Elsevier}
}

@incollection{savic2022heterogeneous,
  title={Heterogeneous tumour modeling using PhysiCell and its implications in precision medicine},
  author={Savi{\'c}, Milo{\v{s}} and Kurbalija, Vladimir and Balaz, Igor and Ivanovi{\'c}, Mirjana},
  booktitle={Cancer, Complexity, Computation},
  pages={157--189},
  year={2022},
  publisher={Springer}
}

@article{chaudhary2025role,
  title={Role of chemokines in aging and age-related diseases},
  author={Chaudhary, Jitendra Kumar and Danga, Ajay Kumar and Kumari, Anita and Bhardwaj, Akshay and Rath, Pramod C},
  journal={Mechanisms of Ageing and Development},
  volume={223},
  pages={112009},
  year={2025},
  publisher={Elsevier}
}

@article{foeng2022harnessing,
  title={Harnessing the chemokine system to home CAR-T cells into solid tumors},
  author={Foeng, Jade and Comerford, Iain and McColl, Shaun R},
  journal={Cell Reports Medicine},
  volume={3},
  number={3},
  year={2022},
  publisher={Elsevier}
}

@article{brummel2023tumour,
  title={Tumour-infiltrating lymphocytes: from prognosis to treatment selection},
  author={Brummel, Koen and Eerkens, Anneke L and de Bruyn, Marco and Nijman, Hans W},
  journal={British journal of cancer},
  volume={128},
  number={3},
  pages={451--458},
  year={2023},
  publisher={Nature Publishing Group UK London}
}

@article{lopez2025biological,
  title={Biological and clinical significance of tumour-infiltrating lymphocytes in the era of immunotherapy: a multidimensional approach},
  author={Lopez de Rodas, Miguel and Villalba-Esparza, Maria and Sanmamed, Miguel F and Chen, Lieping and Rimm, David L and Schalper, Kurt A},
  journal={Nature Reviews Clinical Oncology},
  volume={22},
  number={3},
  pages={163--181},
  year={2025},
  publisher={Nature Publishing Group UK London}
}

@article{yu2024extracellular,
  title={Extracellular matrix stiffness and tumor-associated macrophage polarization: new fields affecting immune exclusion},
  author={Yu, Ke-Xun and Yuan, Wei-Jie and Wang, Hui-Zhen and Li, Yong-Xiang},
  journal={Cancer Immunology, Immunotherapy},
  volume={73},
  number={6},
  pages={115},
  year={2024},
  publisher={Springer}
}

@article{bruni2023cancer,
  title={Cancer immune exclusion: breaking the barricade for a successful immunotherapy},
  author={Bruni, Sofia and Mercogliano, Maria Florencia and Mauro, Florencia Luciana and Cordo Russo, Rosalia In{\'e}s and Schillaci, Roxana},
  journal={Frontiers in Oncology},
  volume={13},
  pages={1135456},
  year={2023},
  publisher={Frontiers Media SA}
}

@article{horstmann20031970,
  title={From 1970 until present: The Keller-Segel model in chemotaxis and its consequences. I},
  author={Horstmann, D},
  journal={Jahresber. Deutsch. Math.-Verein.},
  volume={105},
  pages={103},
  year={2003}
}

@article{arumugam2021keller,
  title={Keller-Segel chemotaxis models: a review},
  author={Arumugam, Gurusamy and Tyagi, Jagmohan},
  journal={Acta Applicandae Mathematicae},
  volume={171},
  number={1},
  pages={6},
  year={2021},
  publisher={Springer}
}

@article{de2005validated,
  title={A validated mathematical model of cell-mediated immune response to tumor growth},
  author={de Pillis, Lisette G and Radunskaya, Ami E and Wiseman, Charles L},
  journal={Cancer research},
  volume={65},
  number={17},
  pages={7950--7958},
  year={2005},
  publisher={American Association for Cancer Research}
}

@article{miller2003autonomous,
  title={Autonomous T cell trafficking examined in vivo with intravital two-photon microscopy},
  author={Miller, Mark J and Wei, Sindy H and Cahalan, Michael D and Parker, Ian},
  journal={Proceedings of the National Academy of Sciences},
  volume={100},
  number={5},
  pages={2604--2609},
  year={2003},
  publisher={The National Academy of Sciences}
}

@article{thurley2015three,
  title={Three-dimensional gradients of cytokine signaling between T cells},
  author={Thurley, Kevin and Gerecht, Daniel and Friedmann, Elfriede and H{\"o}fer, Thomas},
  journal={PLoS computational biology},
  volume={11},
  number={4},
  pages={e1004206},
  year={2015},
  publisher={Public Library of Science San Francisco, CA USA}
}

@article{poznansky2000active,
  title={Active movement of T cells away from a chemokine},
  author={Poznansky, Mark C and Olszak, Ivona T and Foxall, Russell and Evans, Richard H and Luster, Andrew D and Scadden, David T},
  journal={Nature medicine},
  volume={6},
  number={5},
  pages={543--548},
  year={2000},
  publisher={Nature Publishing Group}
}

@article{sutherland1990functional,
  title={Functional characterization of individual human hematopoietic stem cells cultured at limiting dilution on supportive marrow stromal layers.},
  author={Sutherland, Heather J and Lansdorp, Peter M and Henkelman, Don H and Eaves, Allen C and Eaves, Connie J},
  journal={Proceedings of the National Academy of Sciences},
  volume={87},
  number={9},
  pages={3584--3588},
  year={1990}
}

@article{wang2025analysis,
  title={Analysis framework for stochastic predator--prey model with demographic noise},
  author={Wang, Louis Shuo and Yu, Jiguang},
  journal={Results in Applied Mathematics},
  volume={27},
  pages={100621},
  year={2025},
  publisher={Elsevier}
}

@article{wang2025analysis1,
  title={Analysis and mean-field limit of a hybrid pde-abm modeling angiogenesis-regulated resistance evolution},
  author={Wang, Louis Shuo and Yu, Jiguang and Li, Shijia and Liu, Zonghao},
  journal={Mathematics},
  volume={13},
  number={17},
  pages={2898},
  year={2025},
  publisher={MDPI}
}

@article{liang2025global,
  title={Global well-posedness and stability of nonlocal damage-structured lineage model with feedback and dedifferentiation},
  author={Liang, Ye and Wang, Louis Shuo and Yu, Jiguang and Liu, Zonghao},
  journal={Mathematics},
  volume={13},
  number={22},
  pages={3583},
  year={2025},
  publisher={MDPI}
}

@article{wang2026algebraic,
  title={Algebraic--spectral thresholds and discrete--continuous stability transfer in Leslie--Gower systems},
  author={Wang, Louis Shuo and Yu, Jiguang},
  journal={Electronic Research Archive},
  volume={34},
  number={1},
  pages={251--290},
  year={2026},
  publisher={American Institute of Mathematical Sciences (AIMS)}
}

@article{wang2026damage,
  title={A damage-structured PDE model of stem cell hierarchies: The dual role of dedifferentiation in tissue homeostasis and aging},
  author={Wang, Louis Shuo and Yu, Jiguang and Liu, Zonghao},
  journal={Plos one},
  volume={21},
  number={2},
  pages={e0335163},
  year={2026},
  publisher={Public Library of Science San Francisco, CA USA}
}

@article{wang2026breakdown,
  title={The breakdown of linear quasi-cycles: Demographic noise and absorbing boundaries in finite predator--prey systems},
  author={Wang, Louis Shuo and Yu, Jiguang and Liang, Ye and Zhang, Jilin},
  journal={Electronic Research Archive},
  volume={34},
  number={6},
  pages={4248--4289},
  year={2026},
  publisher={American Institute of Mathematical Sciences (AIMS)}
}

@article{yu2026rigorous,
  title={Rigorous analysis of a nonlocal transport--renewal system for physiologically structured populations},
  author={Yu, Jiguang and Wang, Louis Shuo and Liang, Ye},
  journal={Mathematical Methods in the Applied Sciences},
  year={2026}
}

@article{yu2026beyond,
  title={Beyond diagonal noise: A better predator-prey modeling framework with cross-covariance},
  author={Yu, Jiguang and Wang, Louis Shuo},
  journal={Plos one},
  volume={21},
  number={5},
  pages={e0350127},
  year={2026},
  publisher={Public Library of Science San Francisco, CA USA}
}

@article{yu2026microscopic,
  title={From microscopic damage to macroscopic games: a dimensionality reduction of stem cell homeostasis},
  author={Yu, Jiguang and Wang, Louis Shuo and Ban, Shihan and Liang, Ye},
  journal={Transport Phenomena},
  volume={1},
  number={2},
  pages={20260037},
  year={2026},
  publisher={De Gruyter}
}

@article{cai2026optimal,
  title={Optimal harvesting for nonlinear size-structured populations with nonlocal environmental feedback},
  author={Cai, Jie and Chen, Xiaoyang and Gu, Longfei and Chen, Jiayao and Chu, Nuo and Wang, Louis Shuo and Liang, Ye and Yu, Jiguang},
  journal={Mathematics},
  volume={14},
  number={11},
  pages={2025},
  year={2026},
  publisher={MDPI}
}

@article{wang2026elliptic,
  title={Elliptic criticality versus volterra memory in indirect chemotaxis cascades},
  author={Wang, Louis Shuo and Yu, Jiguang and Liang, Ye and Zhang, Jilin},
  journal={Transport Phenomena},
  volume={1},
  number={2},
  pages={20260061},
  year={2026},
  publisher={De Gruyter}
}

@article{liu2025bidirectional,
  title={Bidirectional endothelial feedback drives turing-vascular patterning and drug-resistance niches: a hybrid PDE-agent-based study},
  author={Liu, Zonghao and Wang, Louis Shuo and Yu, Jiguang and Zhang, Jilin and Martel, Erica and Li, Shijia},
  journal={Bioengineering},
  volume={12},
  number={10},
  pages={1097},
  year={2025},
  publisher={MDPI}
}

@article{liang2026separation,
  title={Separation-like irregularity and sample size optimism in high-discrimination logistic prediction models},
  author={Liang, Ye and Wang, Louis Shuo and Yu, Jiguang and Zan, Xin},
  journal={PloS one},
  volume={21},
  number={8},
  pages={e0342286},
  year={2026},
  publisher={Public Library of Science San Francisco, CA USA}
}

@article{yu2026size,
  title={Size-Selective Threshold Harvesting Under Nonlocal Crowding and Exogenous Recruitment},
  author={Yu, Jiguang and Wang, Louis Shuo and Liang, Ye},
  journal={International Journal of Differential Equations},
  volume={2026},
  number={1},
  pages={2523535},
  year={2026},
  publisher={Wiley}
}

@article{yu2026pattern,
  title={Pattern suppression and recovery under one-way versus two-way chemotactic coupling in hybrid partial differential equation--ordinary differential equation models},
  author={Yu, Jiguang and Wang, Louis Shuo and Liu, Zonghao and Liu, Jingfeng},
  journal={Transport Phenomena},
  volume={1},
  number={1},
  pages={20260023},
  year={2026},
  publisher={De Gruyter}
}

@article{yu2026age,
  title={Age-Structured Harvesting Models: A Structural Comparison of Rate-Control and Effort-Control Optimality Systems},
  author={Yu, Jiguang and Wang, Louis Shuo and Liang, Ye},
  journal={International Journal of Mathematics and Mathematical Sciences},
  volume={2026},
  number={1},
  pages={6212516},
  year={2026},
  publisher={Wiley}
}

@article{liu2026ftu,
  title={FTU-Seek: Foundation Model-Guided Hard-Negative Learning for Sparse Functional Tissue Unit Segmentation},
  author={Liu, Zonghao and Su, Lei and Yu, Jiguang and Geng, Xuqing and Wang, Louis Shuo and Wang, Jianmin and Liu, Jingfeng},
  journal={Biomedicines},
  volume={14},
  number={9},
  pages={1935},
  year={2026},
  publisher={Multidisciplinary Digital Publishing Institute}
}

@article{liu2026computational,
  title={Computational Oncology of Chemotaxis-Driven Tumour--Immune Spatial Patterning and Stability},
  author={Liu, Z and Yu, J and Wang, LS and Su, L and Liang, Y and Du, Y and Liu, J},
  journal={Bioengineering},
  volume={13},
  pages={952},
  year={2026},
  publisher={MDPI}
}

@article{yu2026chemotactic,
  title={Chemotactic Feedback Controls Patterning in Hybrid Tumor--Stroma Model},
  author={Yu, J and Wang, LS and Liu, Z and others},
  journal={Bulletin of Mathematical Biology},
  volume={88},
  pages={169},
  year={2026},
  publisher={Springer}
}

@article{yu2026fullcovariance,
  author  = {Yu, Jiguang and Wang, Louis Shuo and Gao, Yuansheng and Liang, Ye},
  title   = {Full-covariance chemical Langevin predator--prey diffusion with absorbing boundaries},
  journal = {Royal Society Open Science},
  volume  = {13},
  number  = {8},
  pages   = {260260},
  year    = {2026},
  month   = {8},
  doi     = {10.1098/rsos.260260}
}

@article{wang2025multi,
  title={Multi-strategy hybrid improved intelligent algorithm for solving uav-mtsp},
  author={Wang, Zixin and Wang, Danqing and Yu, Jiguang},
  journal={Information Technology and Control},
  volume={54},
  number={2},
  pages={413--438},
  year={2025}
}

@inproceedings{gao2022rolling,
  title={Rolling prediction model of closing price based on EEMD data noise reduction and HGS-DELM},
  author={Gao, Yuansheng and Li, Lei and Yu, Jiguang},
  booktitle={2022 International Conference on Data Analytics, Computing and Artificial Intelligence (ICDACAI)},
  pages={255--260},
  year={2022},
  organization={IEEE}
}

\end{document}